\documentclass[11pt]{amsart}
\usepackage{tabularx,booktabs}
\usepackage{caption}
\usepackage{amsmath}
\usepackage{amsfonts}
\usepackage{amscd}
\usepackage{amsthm}
\usepackage{amssymb} \usepackage{latexsym}
\usepackage{eufrak}
\usepackage{euscript}
\usepackage{epsfig}
\usepackage{graphics}
\usepackage{array}
\usepackage{enumerate}
\usepackage{color}
\usepackage{wasysym}
\usepackage{hyperref}
\usepackage{pdfsync}
\usepackage{enumitem}
\usepackage{tikz}
\usepackage{graphicx}
\usepackage{float}
\usepackage{subfigure}
\numberwithin{equation}{section}

\newcommand{\ba}{\begin{eqnarray}}
\newcommand{\ea}{\end{eqnarray}}

\newcommand{\Hmm}[1]{\leavevmode{\marginpar{\tiny%
$\hbox to 0mm{\hspace*{-0.5mm}$\leftarrow$\hss}%
\vcenter{\vrule depth 0.1mm height 0.1mm width \the\marginparwidth}%
\hbox to
0mm{\hss$\rightarrow$\hspace*{-0.5mm}}$\\\relax\raggedright #1}}}

\newtheorem{theorem}{Theorem}[section]
\newtheorem{lemma}[theorem]{Lemma}

\newtheorem{definition}[theorem]{Definition}

\newtheorem{remark}[theorem]{Remark}

\newtheorem{proposition}[theorem]{Proposition}

\newtheorem{example}[theorem]{Example}

\begin{document}
\title[Cheeger inequalities associated with isocapacitary constants]{Cheeger type inequalities associated with isocapacitary constants on graphs}

\author{Bobo Hua}
\address{Bobo Hua: School of Mathematical Sciences, LMNS, Fudan University, Shanghai 200433, China; Shanghai Center for Mathematical Sciences, Fudan University, Shanghai 200433, China}
\email{\href{mailto:bobohua@fudan.edu.cn}{bobohua@fudan.edu.cn}}

\author{Florentin M\"unch}
\address{Florentin M\"unch: Max Planck Institute for Mathematics in the Sciences, Leipzig 04103, Germany}
\email{\href{mailto:cfmuench@gmail.com}{cfmuench@gmail.com}}

\author{Tao Wang}
\address{Tao Wang: School of Mathematical Sciences, Fudan University, Shanghai 200433, China}
\email{\href{mailto:taowang21@m.fudan.edu.cn}{taowang21@m.fudan.edu.cn}}

\maketitle

\begin{abstract}
  In this paper, we introduce Cheeger type constants via isocapacitary constants introduced by Maz'ya to estimate first Dirichlet, Neumann and Steklov eigenvalues on a finite subgraph of a graph. Moreover, we estimate the bottom of the spectrum of the Laplace operator and the Dirichlet-to-Neumann operator for an infinite subgraph. Estimates for higher-order Steklov eigenvalues on a finite or infinite subgraph are also proved.
\end{abstract}

\noindent\textbf{Mathematics Subject Classification: }53A70, 05C50, 15A42, 39A12

\bigskip

\section{Introduction}

Let $(M, g)$ be a compact, connected, smooth Riemannian manifold with smooth boundary $\partial M$. There are three typical eigenvalue problems: the Dirichlet problem, Neumann problem and Steklov problem. Eigenvalue estimates are of interest in spectral geometry. In \cite{Cheeger1970}, Cheeger discovered a close relation between the first non-trivial eigenvalue of the Laplace-Beltrami operator on a closed manifold and the isoperimetric constant, called the Cheeger constant. Estimates of this type are called Cheeger estimates. 

For a compact manifold $M$ with boundary $\partial M$, the Steklov problem is to study the eigenvalues of the Dirichlet-to-Neumann operator $\Lambda$, called the DtN operator for short, which is defined as
\begin{align*}
 \Lambda: H^{\frac12}(\partial M) &\to H^{-\frac12}(\partial M), \\
 f &\mapsto \Lambda (f):= \frac{\partial u_f}{\partial n},
\end{align*}
where $u_f$ is the harmonic extension of $f$ to $M$, $n$ is the outward unit normal on the boundary $\partial M$ and $H^{\frac12}(\partial M)$, $H^{-\frac12}(\partial M)$ are Sobolev spaces. The DtN operator is closely related to the Calder\'on problem \cite{Calder1980} of determining the anisotropic electrical conductivity of a medium in the Euclidean space by taking voltage and current measurements at the boundary of the medium. This makes it useful for applications to electrical impedance tomography, see \cite{LU2001, Uhlmann2014} for more details.

The Steklov problem has been extensively studied on Riemannian manifolds, see e.g. \cite{GL2021, CEG2011, Escobar1999, FS2011, FS2016, Hassannezhad2011}. We refer to a comprehensive survey \cite{CGGS2024}. Especially, Escobar \cite{Escobar1997} and Jammes \cite{Jammes2015} introduced two different types of Cheeger estimates for the first non-trivial eigenvalue of the DtN operator and proved corresponding Cheeger estimates. 

In recent years, many mathematicians are interested in the analysis on graphs. Spectral graph theory is an important field in the literature \cite{BH2012, CDS1995, Chung1997, HH2023, HH2022, Fuji1996, MW1989, Mohar1982, Perrin2019, Chung1989, RVW2002, H2019, MSS2015, JTYZZ2021}. The Cheeger estimate was first generalized to graphs by Dodziuk \cite{Dodziuk1984}, Alon and Milman \cite{AM1985}, respectively. Inspired by their works, there were various Cheeger estimates on graphs, see e.g. \cite{BHJ2014, DK1986, BKW2015, Liu2015, TH2018, HH2018, LGT2014, Perrin2021, HW2020}.

The first author, Huang and Wang \cite{HHW2017} defined the DtN operator on a finite subgraph of a graph and proved two types of Cheeger estimates similar to \cite{Escobar1997, Jammes2015} for the first non-trivial Steklov eigenvalue. Hassannezhad and Miclo \cite{HM2020} independently proved the Cheeger estimate similar to \cite{Jammes2015} and generalized it to higher-order Cheeger estimates for the Steklov problem. In \cite{HHW2022}, the first author, Huang and Wang used the exhaustion methods to define the DtN operator on an infinite subgraph of a graph and proved the Cheeger estimate for the bottom of the spectrum and higher-order Cheeger estimates for higher-order eigenvalues of the DtN operator.

In this paper, we introduce Cheeger constants using the capacity to estimate first non-trivial Dirichlet, Neumann and Steklov eigenvalues of graphs. Using capacity to estimate eigenvalues can be traced back to the works of Maz'ya \cite{M1964, M2009, M1962}. Maz'ya proved that if $\Omega$ is a subdomain of a $n$-dimensional Riemannian manifold $M$ and $\lambda_1(\Omega)$ is the first Dirichlet eigenvalue, then
\[
 \frac14\Gamma(\Omega) \leq \lambda_1(\Omega) \leq \Gamma(\Omega),
\]
where 
\[
 \Gamma(\Omega):= \inf_{F \subset \subset \Omega}\frac{\mathrm{Cap}(F, \Omega)}{|F|},
\]
with
\[
 \mathrm{Cap}(F, \Omega):= \inf\left\{\int_{\Omega}|\nabla u|^2 dx: u \in C^{\infty}_0(\Omega), u \geq 1 \text{ on }F\right\}.
\]
We will prove similar estimates on graphs, and apply them to the Steklov problem.

We recall some basic definitions on graphs. Let $G = (V, E, w, m)$ be a (possibly infinite) undirected, simple graph with the set of vertices $V$, the set of edges $E$, the edge weight $w: E \to \mathbb{R}_+$ and the vertex weight $m: V \to \mathbb{R}_+$. We write $m(A):= \sum_{x \in A}m(x)$ for the volume of any subset $A \subseteq V$. Two vertices $x, y$ are called neighbors, denoted by $x \sim y$, if there is an edge connecting $x$ and $y$, i.e., $\{x, y\} \in E$. We only consider locally finite graphs, i.e., each vertex only has finitely many neighbors. A graph is called connected if for any $x, y \in V$ there is a path $\{z_i\}_{i = 0}^n \subseteq V$ connecting $x$ and $y$, i.e.,
\[
 x = z_0 \sim z_1 \sim \cdots \sim z_n = y.
\]
We call the quadruple $G = (V, E, w, m)$ a weighted graph.

For any subset $\Omega \subseteq V$, we define the vertex boundary of $\Omega$ as
\[
 \delta \Omega:= \{y \in V \setminus \Omega: \text{there exists } x \in \Omega \text{ such that } x \sim y\}.
\]
Set $\overline{\Omega}:= \Omega \cup \delta \Omega$. In this paper, we always assume that $\overline{\Omega}$ is connected. Given $\Omega_1, \Omega_2 \subseteq V$, the set of edges between $\Omega_1$ and $\Omega_2$ is written as
\[
 E(\Omega_1, \Omega_2):= \{\{x, y\} \in E: x \in \Omega_1, y \in \Omega_2\}.
\]
Given any set $F$, we denote by $\mathbb{R}^{F}$ the set of all real functions defined on $F$. For any subset $\Omega \subseteq V$ and any $f \in \mathbb{R}^{\overline{\Omega}}$, the Laplacian of $f$ is defined as
\[
 \Delta f(x):= \frac{1}{m(x)}\sum_{y \sim x}w(x, y)(f(y) - f(x)), \quad x \in \Omega,
\]
and the outward normal derivative of $f$ at $z \in \delta \Omega$ is defined as
\[
 \frac{\partial f}{\partial n}(z):= \frac{1}{m(z)}\sum_{x \in \Omega: x \sim z}w(x, z)(f(z) - f(x)).
\]
We write $\mathcal{L}:= -\Delta$. For any subset $X \subseteq V$, we restrict $w$ to $E(X, \overline{X})$ and $m$ to $\overline{X}$, still denoted by $w$ and $m$ for simplicity. Then we define a graph
\begin{equation}\label{DefOfUsedGraph}
 G_{X} = (\overline{X}, E(X, \overline{X}), w, m).
\end{equation}
Note that edges between vertices in $\delta X$, $E(\delta X, \delta X)$, are removed, i.e., $w(x, y) = 0$ for any $\{x, y\} \notin E(X, \overline{X})$.

In analogy to the continuous case, one can study three eigenvalue problems in the discrete setting. For any finite subset $\Omega \subseteq V$, the Dirichlet problem on $\Omega$ is defined as 
\begin{equation*}
 \begin{cases}
  \mathcal{L}f(x) = \lambda f(x), \quad &x \in \Omega, \\
  f(x) = 0, &x \in \delta \Omega.
 \end{cases}
\end{equation*}
We denote by $\lambda_1^D(\Omega)$ the first eigenvalue of the above Dirichlet problem. Set
\[
 \alpha_D(\Omega):= \inf_{A \subseteq \Omega}\frac{\mathrm{Cap}_{\Omega}(A)}{m(A)},
\]
where $\mathrm{Cap}_{\Omega}(A)$ is defined in \eqref{DefOfCap2} with respect to the graph $G_{\Omega}$. We prove the following theorem.
\begin{theorem}\label{CheegerInequalityForDirichlet}
 Let $G$ be a weighted graph, $\Omega \subseteq V$ be a finite subset. Then
 \begin{equation}\label{DirichletEigenvalue}
  \frac14\alpha_D(\Omega) \leq \lambda_1^D(\Omega) \leq \alpha_D(\Omega).
 \end{equation}
\end{theorem}
\begin{remark} 
\
 \begin{itemize}
  \item[1.] Our estimate shows that $\lambda_1^D(\Omega)$ and $\alpha_D(\Omega)$ are of the same order, which is better than the classical Cheeger estimate for normalized Laplacian $\frac{h^2(\Omega)}{2} \leq \lambda_1^D(\Omega) \leq h(\Omega)$, see e.g., \cite{Dodziuk1984, BHJ2014}, where $h(\Omega)$ is the Cheeger constant.
  \item[2.] The Cheeger estimate for general Laplacian depends on $\max\limits_{x \in \Omega}\mathrm{Deg}(x)$, where $\mathrm{Deg}(x):= \frac{1}{m(x)}\sum_{y \sim x}w(x, y)$ is the degree of vertex $x$. But the constant in the estimate \eqref{DirichletEigenvalue} does not depend on weights. This is the key for our application to prove estimate of the first Steklov eigenvalue.
 \end{itemize}
\end{remark}
Next we use \eqref{DirichletEigenvalue} to estimate the bottom of the spectrum of an infinite graph $G$. Let $\{W_i\}_{i=1}^{\infty}$ be an exhaustion of $G$, see Definition \ref{DefinitionOfExhaustion}. By the spectral theory, the bottom of the spectrum of the Laplacian on $G$ is given by
\[
 \lambda_1(G)= \lim_{i \to \infty}\lambda_1^D(W_i).
\]
Set
\[
 \alpha_D(G):= \inf_{\substack{A \subseteq V \\ |A| < +\infty}}\frac{\mathrm{Cap}(A)}{m(A)},
\]
where $$\mathrm{Cap}(A)= \lim_{i \to \infty}\mathrm{Cap}_{W_i}(A).$$ Then we have the following theorem.
\begin{theorem}\label{CheegerInequalityForBottom}
 For an infinite weighted graph $G$,
 \[
  \frac14\alpha_D(G) \leq \lambda_1(G) \leq \alpha_D(G).
 \]
\end{theorem}

As a corollary, we prove that $\lambda_1(G) = 0$ for a recurrent infinite weighted graph $G$, see Proposition \ref{ZeroBottomAndRecurrent}. For the transient case, we give an example of estimating the bottom of the spectrum which was considered in \cite{HKSW2023}, see Example \ref{TransientCase}.

For any finite subset $\Omega \subseteq V$, the Neumann problem on $\Omega$ is defined as 
\begin{equation*}
 \begin{cases}
  \mathcal{L}f(x) = \lambda f(x), \quad &x \in \Omega, \\
  \frac{\partial f}{\partial n}(x) = 0, &x \in \delta \Omega.
 \end{cases}
\end{equation*}
We denote by $\lambda_1^N(\Omega)$ the first non-trivial eigenvalue of the above Neumann problem. For
\[
 \alpha_N(\Omega):= \inf_{A, B \subseteq \Omega}\frac{\mathrm{Cap}_{\Omega}(A, B)}{m(A) \wedge m(B)},
\]
where $a \wedge b:= \min\{a, b\}$ and $\mathrm{Cap}_{\Omega}(A, B)$ is defined in \eqref{DefOfCap}, we have the following refined estimate, compared to \cite{HH2018}.
\begin{theorem}\label{CheegerInequalityForNeumann}
 Let $G$ be a weighted graph, $\Omega \subseteq V$ be a finite subset with $|\Omega| \geq 2$,  where $|\cdot|$ denotes the cardinality of a set. Then
 \[
  \frac18\alpha_N(\Omega) \leq \lambda_1^N(\Omega) \leq 2\alpha_N(\Omega).
 \]
\end{theorem}

For any finite subset $\Omega \subseteq V$, we define the DtN operator on $\Omega$, denoted by $\Lambda_{\Omega}$, as
\begin{align*}
 \Lambda_{\Omega}: \mathbb{R}^{\delta \Omega} &\to \mathbb{R}^{\delta \Omega} \\
 f &\mapsto \Lambda_{\Omega}f:= \frac{\partial u_f}{\partial n},
\end{align*}
where $u_f$ is the harmonic extension of $f$ to $\Omega$. The Steklov problem on $\Omega$ is defined as
\begin{equation*}
 \begin{cases}
  \mathcal{L}f(x) = 0, \quad &x \in \Omega, \\
  \frac{\partial f}{\partial n}(x) = \lambda f(x), &x \in \delta \Omega.
 \end{cases}
\end{equation*}
We denote by $\sigma_1(\Omega)$ the first non-trivial Steklov eigenvalue. For $|\delta \Omega| \geq 2$, set
\[
 \alpha_S(\Omega):= \inf_{A, B \subseteq \delta \Omega}\frac{\mathrm{Cap}_{\Omega}(A, B)}{m(A) \wedge m(B)}.
\]
We have the following estimate.
\begin{theorem}\label{CheegerInequalityForSteklov}
 Let $G$ be a weighted graph, $\Omega \subseteq V$ be a finite subset with $|\delta \Omega| \geq 2$. Then
 \begin{equation}\label{CIFS}
  \frac18\alpha_S(\Omega) \leq \sigma_1(\Omega) \leq 2\alpha_S(\Omega).
 \end{equation}
\end{theorem}

\begin{remark}
\
 \begin{itemize}
  \item[1.] The definition of our DtN operator is slightly different from that in \cite{HM2020}. For a subset $\Omega \subseteq V$, we always consider the Neumann problem and the Steklov problem on the graph $G_{\Omega}$ defined in \eqref{DefOfUsedGraph}. But similar estimates to \eqref{CIFS} also hold for the first non-trivial eigenvalue of the DtN operator in \cite{HM2020}, see Section 7.
  \item[2.] Cheeger estimates for the Steklov problem proved in \cite{HHW2017, HM2020, HHW2022} involve two Cheeger constants $h(\Omega)$ and $h_J(\Omega)$. Theorem \ref{CheegerInequalityForSteklov} provides a single geometric quantify $\alpha_S(\Omega)$ for estimating $\sigma_1(\Omega)$, which are of same order.
  \item[3.] The Cheeger estimate in \cite{HHW2017} $\sigma_1(\Omega) \leq h_J(\Omega)$ may not be effective for some graphs. For example, for the graph of infinite line $(\mathbb{Z}, E)$ with unit edge weights and vertex weights, let $\Omega = \{1, \cdots, n-1\}$. Then for any $n \geq 2$, the Cheeger estimate proved in \cite{HHW2017} implies a trivial  upper bound estimate $\sigma_1(\Omega) \leq 1$. By applying \eqref{CIFS}, we can get a sharp upper bound estimate $\sigma_1(\Omega) \leq \frac{2}{n}$. See Example \ref{Line}. Moreover, we characterize the equality case for the upper bound estimate in \eqref{CIFS} in Proposition \ref{EqualityCase}.
  \item[4.] For a finite tree $T$ with unit edge weights and vertex weights, we regard the vertices of degree one as boundary vertices. Then the upper bound estimate $\sigma_1(T) \leq \frac{2}{\mathrm{diam}(T)}$ is proved by constructing test functions in \cite{HH2022}, where $\mathrm{diam}(T):= \sup_{x, y \in V}d(x, y)$ is the diameter of $T$. This estimate can be easily proved by applying \eqref{CIFS}. See Example \ref{FiniteTree}.
 \end{itemize}
\end{remark}

The proofs of Theorem \ref{CheegerInequalityForNeumann} and Theorem \ref{CheegerInequalityForSteklov} follow a same strategy. Consider the graph $G_{\Omega}$. We use co-area formula, see Lemma \ref{CoareaFormula}, to prove that the first non-trivial eigenvalue $\lambda_1^{\widetilde{\mathcal{L}}}$ of the Laplacian on $G_{\Omega}$ satisfies
\[
 \frac18\widetilde{\alpha}_{\widetilde{\mathcal{L}}} \leq \lambda_1^{\widetilde{\mathcal{L}}} \leq 2\widetilde{\alpha}_{\widetilde{\mathcal{L}}},
\]
where
\[
 \widetilde{\alpha}_{\widetilde{\mathcal{L}}}:= \inf_{A, B \subseteq \overline{\Omega}}\frac{\mathrm{Cap}_{\Omega}(A, B)}{m(A) \wedge m(B)}.
\]
Then we construct a sequence of finite graphs $\{G^{(k)}\}_{k=1}^{\infty}$ where $G^{(1)} = G_{\Omega}$, $G^{(k)} = (\overline{\Omega}, E(\Omega, \overline{\Omega}), w, m^{(k)})$ with $m^{(k)}|_{\delta \Omega}= m$ and $m^{(k)}|_{\Omega} = \frac{m}{k}$. Following the proof in \cite{HM2020}, the first non-trivial eigenvalue $\lambda_1^{(k)}$ of the Laplacian on $G^{(k)}$ satisfies
\[
 \lim_{k \to \infty}\lambda_1^{(k)} = \sigma_1(\Omega).
\]
Moreover, one can easily show that the isocapacitary constant converges to $\alpha_S(\Omega)$. This proves Theorem \ref{CheegerInequalityForSteklov} by passing to the limit $k \to +\infty$.

We can estimate Steklov eigenvalues of an infinite subset of $G$. The DtN operator on an infinite subset was introduced in \cite{HM2020, HHW2022} with same restriction on vertex weights. We introduce a generalization of the DtN operator on an infinite subset of arbitrary vertex weights, see Section 5 for details. Let $U \subseteq V$ be an infinite subset with vertex boundary $\delta U$. We always assume that $\overline{U}$ is connected. Denote by $\sigma_1(U)$ the bottom of the spectrum of the DtN operator and set
\[
 \alpha_S(U):= \inf_{\substack{A \subseteq \delta U \\ |A| < + \infty}}\frac{\mathrm{Cap}^U(A)}{m(A)},
\]
where $\mathrm{Cap}^U(A)$ is defined in \eqref{DefOfCap3}. We have the following theorem.
\begin{theorem}\label{CheegerInequalityForDtNBottom}
 Let $G$ be an infinite weighted graph, $U \subseteq V$ be an infinite subset. Then
 \[
  \frac14\alpha_S(U) \leq \sigma_1(U) \leq \alpha_S(U).
 \]
\end{theorem}

As an application, we show that for $N \geq 3$, the bottom of the spectrum of the DtN operator on $\mathbb{Z}^N_+:= \left\{(x_1, \cdots, x_N): x_1, \cdots, x_{N-1} \in \mathbb{Z}, x_N \in \mathbb{Z}_+\right\}$ is zero, see Example \ref{HalfOfZn}, which seems hard to prove using the upper bound $h_J$ in \cite{HHW2022}. 

\begin{remark}
 Compared with the estimate in \cite{HHW2022}, our estimate only use one geometric quantity. In particular, our upper bound estimate is sharp, while the upper bound estimate $h_J$ in \cite{HHW2022} is not. See Example \ref{InfiniteTree} and \cite[Example 4.7]{HHW2022}.
\end{remark}

Inspired by \cite{LGT2014, HM2020, HHW2022}, we can estimate higher-order Dirichlet eigenvalues and Steklov eigenvalues. For any subset $W \subseteq V$, we denote by $\mathcal{A}(W)$ the collection of all non-empty finite subsets of $W$ and $\mathcal{A}_k(W)$ the set of all disjoint $k$-tuples $(A_1, \cdots, A_k)$ such that $A_l \in \mathcal{A}(W)$, for all $l \in \{1, \cdots, k\}$.

For an infinite weighted graph $G = (V, E, w, m)$ and $k \geq 1$, we denote by $\lambda_k(G)$ the $k$-th variational eigenvalue of the Laplacian, see \eqref{KthEigenvalueOfLaplacian}. Set
\[
 \Gamma_k^D(G):= \inf_{(A_1, \cdots, A_k) \in \mathcal{A}_k(V)}\max_{l \in \{1, \cdots, k\}}\alpha_D(A_l).
\]
Then we get the following theorem.
\begin{theorem}\label{HigherOrderDirichletEstimate}
 For an infinite weighted graph $G$, there exists a universal constant $c>0$ such that for any $k \geq 1$,
 \[
  \frac{c}{k^6}\Gamma_k^D(G) \leq \lambda_k(G) \leq 2\Gamma_k^D(G).
 \]
\end{theorem}
We refer to \cite[Section 5]{HHW2022} for higher-order estimates for Dirichlet eigenvalues on a finite graph, see also Lemma \ref{HighOrderDirichlet}.
 
For any subset $X \subseteq V$ and $Y \in \mathcal{A}(\overline{X})$, we define
 \[
  \alpha_{DS}^X(Y) := \inf_{A \subseteq Y \cap \delta X} \frac{\mathrm{Cap}_{X}(A, \delta_XY)}{m(A)},
 \]
where $\delta_XY$ is the vertex boundary of $Y$ in $G_X$. By convention, when $Y \cap \delta X = \varnothing$, we set $\alpha_{DS}^X(Y) = +\infty$.

For any finite subset $\Omega \subseteq V$ and any $1 \leq k \leq |\delta \Omega|-1$, we denote by $\sigma_k(\Omega)$ the $k$-th non-trivial Steklov eigenvalue. Set
\[
 \kappa_{k+1}:= \min_{(A_1, \cdots, A_{k+1})\in \mathcal{A}_{k+1}(\overline{\Omega})}\max_{l \in \{1, \cdots, k+1\}}\alpha_{DS}^{\Omega}(A_l),
\]
Then we have the following theorem.
\begin{theorem}\label{HighOrderSteklovOnFinite}
 Let $G$ be a weighted graph, $\Omega \subseteq V$ be a finite subset. Then there exists a universal constant $c>0$ such that for any $1 \leq k \leq |\delta \Omega|-1$,
 \[
  \frac{c}{k^6}\kappa_{k+1} \leq \sigma_k(\Omega) \leq 2\kappa_{k+1}.
 \]
\end{theorem}

For an infinite subset $U \subseteq V$ and any $k \geq 1$, we denote by $\sigma_k(U)$ the $k$-th variational Steklov eigenvalue, see \eqref{DefinitionOfSigmaK} for the definition. Set
\[
 \Gamma^S_k(U):= \inf_{(A_1, \cdots, A_k) \in \mathcal{A}_k(\overline{U})}\max_{l \in \{1, \cdots, k\}}\alpha_{DS}^U(A_l),
\]
then we get the following estimate.
\begin{theorem}\label{HighOrderSteklovOnInfinite}
 Let $G$ be a weighted graph, $U \subseteq V$ be an infinite subset. Then there exists a universal constant $c>0$ such that for any $k \geq 1$,
 \[
  \frac{c}{k^6}\Gamma^S_k(U) \leq \sigma_k(U) \leq 2\Gamma^S_k(U).
 \]
\end{theorem}

The paper is organized as follows. In Section 2, we recall some basic facts on graphs. In Section 3, we study the first Dirichlet eigenvalue on a finite graph and prove Theorem \ref{CheegerInequalityForDirichlet} and Theorem \ref{CheegerInequalityForBottom}. In Section 4, we obtain the estimate of the first non-trivial Neumann eigenvalue on a finite graph, i.e., Theorem \ref{CheegerInequalityForNeumann}. In Section 5, we study the first non-trivial Steklov eigenvalue on a finite graph and prove Theorem \ref{CheegerInequalityForSteklov} and Theorem \ref{CheegerInequalityForDtNBottom}. In Section 6, we prove higher-order eigenvalue estimates Theorem \ref{HigherOrderDirichletEstimate}, Theorem \ref{HighOrderSteklovOnFinite} and Theorem \ref{HighOrderSteklovOnInfinite}. In Section 7, we discuss the eigenvalue problem of the DtN operator defined in \cite{HM2020}.

\section{Preliminaries}
Let $G = (V, E, w, m)$ be a weighted graph, and $F \subseteq V$ be a finite subset. Given a function $f \in \mathbb{R}^F$, we denote by
\[
 \|f\|_{l^2(F)}:= \left(\sum_{x \in F}|f(x)|^2m(x)\right)^{\frac12}
\]
the $l^2$ norm of $f$, and 
\[
 \|f\|_{l^{\infty}(F)}:= \sup_{x \in F}|f(x)|
\]
the $l^{\infty}$ norm of $f$. Let 
\[
 l^2(F):= \{f \in \mathbb{R}^{F}: \|f\|_{l^2(F)} < +\infty\}
\]
be the space of $l^2$ summable functions on $F$. Then $l^2(F)$ is a Hilbert space equipped with standard inner product
\[
 \langle f, g \rangle_F = \sum_{x \in F}f(x)g(x)m(x), \quad \forall \ f, g \in \mathbb{R}^{F}.
\]

For any subset $X \subseteq V$, we denote by $l_0(X)$ the set of functions on $X$ with finite support. Consider the graph $G_X$. Recall that edges between vertices in $\delta X$ are removed. Given functions $f, g \in \mathbb{R}^{\overline{X}}$, we define the energy functional as
\begin{align*}
 \mathcal{E}_{X}(f, g)&:= \frac12\sum_{x, y \in \overline{X}}w(x, y)(f(y)-f(x))(g(y)-g(x)) \\
 &=\sum_{\{x, y\} \in E(X, \overline{X})}w(x, y)(f(y)-f(x))(g(y)-g(x)),
\end{align*}
 whenever the summation absolutely converges. For any finite $A, B \subseteq \overline{X}$, we define
\begin{equation}\label{DefOfCap}
 \mathrm{Cap}_{X}(A, B):= \inf\{\mathcal{E}_{X}(f, f): f|_A = 1, f|_B = 0, f \in l_0(\overline{X})\}.
\end{equation}
For finite $X$, one easily checks that the optimizer $f$ exists, which satisfies $\Delta f = 0$ on $X \setminus (\mathring{\delta}A \cup \mathring{\delta}B)$ and $\frac{\partial f}{\partial n} = 0$ on $\delta X \setminus (\mathring{\delta}A \cup \mathring{\delta}B)$, where 
\[
  \mathring{\delta} A = \{x \in A: \text{there exists } y \in \overline{X} \setminus A \text{ such that } y \sim x \text{ in } G_X\}.
\]
When $X = V$, we write
\[
 \mathcal{E}(f, g) = \mathcal{E}_V(f, g)
\]
and
\[
 \mathrm{Cap}(A, B) = \mathrm{Cap}_V(A, B).
\]
For finite $\Omega \subseteq V$, we set
\begin{equation}\label{DefOfCap2}
 \mathrm{Cap}_{\Omega}(A) = \mathrm{Cap}_{\Omega}(A, \delta \Omega).
\end{equation}

For an infinite graph, the exhaustion of the whole graph by finite subsets of vertices is an important concept, see e.g., \cite{BHJ2014}. 
\begin{definition}\label{DefinitionOfExhaustion}
Let $G = (V, E, w, m)$ be an infinite weighted graph. A sequence of subsets of vertices $\mathcal{W} = \{W_i\}_{i = 1}^{\infty}$ is called an exhaustion of $G$, denoted by $\mathcal{W} \uparrow V$, if it satisfies 
\begin{itemize}
 \item $W_1 \subseteq W_2 \subseteq \cdots \subseteq W_i \subseteq \cdots \subseteq V$,
 \item $|W_i| < \infty$, for all $i = 1, 2, \cdots$,
 \item $V = \bigcup_{i=1}^{\infty}W_i$.
\end{itemize}
\end{definition}
Let $U \subseteq V$ be an infinite subset with vertex boundary $\delta U$. For any finite $A \subseteq \overline{U}$, we define
\begin{equation}\label{DefOfCap3}
 \mathrm{Cap}^U(A):= \inf\{\mathcal{E}_U(f, f): f|_A = 1, f \in l_0(\overline{U})\}.
\end{equation}
For any exhaustion $\mathcal{W} = \{W_i\}_{i=1}^{\infty} \uparrow \overline{U}$, one can show that 
\[
 \mathrm{Cap}^U(A) = \lim_{i \to \infty}\mathrm{Cap}_{U}(A, \delta_UW_i).
\]
When $U = V$, we write
\[
 \mathrm{Cap}(A) = \mathrm{Cap}^V(A),
\]
and we have for any exhaustion $\mathcal{W} = \{W_i\}_{i = 1}^{\infty}$ of $G$,
\[
 \mathrm{Cap}(A) = \lim_{i \to \infty}\mathrm{Cap}_{W_i}(A).
\]

We recall the following well-known result of Green's formula.
\begin{lemma}
 For any finite subset $\Omega \subseteq V$ and any $f, g \in \mathbb{R}^{\overline{\Omega}}$, we have
 \begin{equation}\label{GreenFormula}
  \langle \mathcal{L}f, g \rangle_{\Omega} + \langle \frac{\partial f}{\partial n}, g \rangle_{\delta \Omega} = \mathcal{E}_{\Omega}(f, g).
 \end{equation}
\end{lemma}

\section{First Dirichlet Eigenvalue}

Let $G = (V, E, w, m)$ be a graph, $\Omega \subseteq V$ be a finite subset. Let $\lambda_1^D(\Omega)$ be the first Dirichlet eigenvalue on $\Omega$. By the well-known Rayleigh quotient characterization,
\[
 \lambda_1^D(\Omega) = \inf_{\substack{f \in \mathbb{R}^{\overline\Omega}, f|_{\delta \Omega}=0 \\ f \not\equiv 0}}\frac{\mathcal{E}_{\Omega}(f, f)}{\|f\|_{l^2(\Omega)}^2}.
\]
Moreover, we know that if $f$ is a first eigenfunction, then $|f|>0$ on $\Omega$, see, e.g., \cite{Dodziuk1984}.

The following co-area formula proved in \cite{SS2019} plays a crucial role in proving Theorem \ref{CheegerInequalityForDirichlet}. For the convenience of readers, we give a proof here.
\begin{lemma}(\cite{SS2019})\label{CoareaFormula}
 For any $f \in \mathbb{R}^{\overline{\Omega}}$, we have
 \[
  2\mathcal{E}_{\Omega}(f, f) \geq \int_0^{\infty}t\mathrm{Cap}_{\Omega}(\{f>t\}, \{f \leq 0\})dt.
 \]
\end{lemma}
\begin{proof}
 For any $t > 0$, there exists a function $h_t$ satisfying $h_t|_{\{f > t\}} = 1$ and $h_t|_{\{f \leq 0\}} = 0$ , such that
 \[
  \mathcal{E}_{\Omega}(h_t, h_t) = \mathrm{Cap}_{\Omega}(\{f > t\}, \{f \leq 0\}).
 \]
 Then we have
 \begin{align*}
  t\mathrm{Cap}_{\Omega}(\{f>t\}, \{f \leq 0\}) &= t\mathcal{E}_{\Omega}(h_t, h_t) \\
  &= \langle \mathcal{L}h_t, th_t \rangle_{\Omega} + \langle \frac{\partial h_t}{\partial n}, th_t \rangle_{\delta \Omega} \\
  &\leq \langle f, \mathcal{L}h_t \rangle_{\Omega} + \langle \frac{\partial h_t}{\partial n}, f \rangle_{\delta \Omega}\\
  &= \mathcal{E}_{\Omega}(h_t, f),
 \end{align*}
 where we used \eqref{GreenFormula} twice. Since $\int_0^{\infty} \mathcal{E}_{\Omega}(h_t, f)dt = \mathcal{E}_{\Omega}(f, \int_0^{\infty}h_t dt)$, we get
 \begin{align*}
  \left(\int_0^{\infty} t\mathrm{Cap}_{\Omega}(\{f>t\}, \{f \leq 0\})dt\right)^2 &\leq \left(\mathcal{E}_{\Omega}(f, \int_0^{\infty}h_t dt)\right)^2 \\
  &\leq \mathcal{E}_{\Omega}(f, f) \mathcal{E}_{\Omega}\left(\int_0^{\infty} h_t dt, \int_0^{\infty} h_t dt\right).
 \end{align*}
 By definition of the energy functional, we have
 \begin{align*}
  \mathcal{E}_{\Omega}\left(\int_0^{\infty} h_t dt, \int_0^{\infty} h_t dt\right) &= \int_0^{\infty}\int_0^{\infty} \mathcal{E}_{\Omega}(h_s, h_t) dsdt \\
  &= 2\int_0^{\infty}\int_0^t \mathcal{E}_{\Omega}(h_s, h_t) dsdt.  
 \end{align*}
 Note that when $t > s$, 
 \begin{align*}
  \mathcal{E}_{\Omega}(h_s, h_t) &= \langle \mathcal{L}h_t, h_s \rangle_{\Omega} + \langle \frac{\partial h_t}{\partial n}, h_s \rangle_{\delta \Omega} \\
  &= \sum_{x \in \mathring{\delta} \{f>t\} \cap \Omega}\mathcal{L}h_t(x)h_s(x)m(x) + \sum_{x \in \mathring{\delta}\{f > t\} \cap \delta \Omega}\frac{\partial h_t}{\partial n}(x)h_s(x)m(x) \\
  &=\sum_{x \in \mathring{\delta} \{f>t\} \cap \Omega}\mathcal{L}h_t(x)h_t(x)m(x) + \sum_{x \in \mathring{\delta}\{f>t\} \cap \delta \Omega}\frac{\partial h_t}{\partial n}(x)h_t(x)m(x) \\
  &=\langle \mathcal{L}h_t, h_t \rangle_{\Omega} + \langle \frac{\partial h_t}{\partial n}, h_t \rangle_{\delta \Omega}= \mathcal{E}_{\Omega}(h_t, h_t).
 \end{align*}
 Thus,
 \[
  \mathcal{E}_{\Omega}\left(\int_0^{\infty} h_t dt, \int_0^{\infty} h_t dt\right) = 2\int_0^{\infty} t\mathrm{Cap}_{\Omega}(\{f>t\}, \{f \leq 0\})dt,  
 \]
 and the conclusion follows.
\end{proof}
Now, we are ready to prove Theorem \ref{CheegerInequalityForDirichlet}
\begin{proof}[Proof of Theorem \ref{CheegerInequalityForDirichlet}]
 We first show that $\lambda_1^D(\Omega) \leq \alpha_D(\Omega)$. Recall that 
 \[
  \alpha_D(\Omega) = \inf_{A \subseteq \Omega} \frac{\mathrm{Cap}_{\Omega}(A)}{m(A)}. 
 \]
 Let $A \subseteq \Omega$ be a finite subset such that 
 \[
  \alpha_D(\Omega) = \frac{\mathrm{Cap}_{\Omega}(A)}{m(A)}.
 \]
 Then there exists a function $f$ satisfying $f|_A = 1$ and $f|_{\delta\Omega} =0$, such that $\mathcal{E}_{\Omega}(f, f) = \mathrm{Cap}_{\Omega}(A)$. Note that $f$ can be used as a test function in the Rayleigh quotient characterization. Since $f|_A=1$, we have
 \[
  \|f\|_{l^2(\Omega)}^2 = \sum_{x \in \Omega} f^2(x)m(x) \geq \sum_{x \in A} m(x) = m(A).
 \]
 It follows that
 \[
  \alpha_D(\Omega) = \frac{\mathcal{E}_{\Omega}(f, f)}{m(A)} \geq \frac{\mathcal{E}_{\Omega}(f, f)}{\|f\|_{l^2(\Omega)}^2}\geq \lambda_1^D(\Omega).
 \]
 To show that $\frac14\alpha_D(\Omega) \leq \lambda_1^D(\Omega)$, let $\phi > 0$ be a first eigenfunction. Then by Lemma \ref{CoareaFormula}, noting that $\{\phi \leq 0\} = \delta \Omega$, we have
 \begin{align*}
  2\lambda_1^D(\Omega)\|\phi\|_{l^2(\Omega)}^2 &= 2\mathcal{E}_{\Omega}(\phi, \phi) \geq \int_0^{\infty} t\mathrm{Cap}_{\Omega}(\{\phi > t\}, \{\phi \leq 0\})dt \\
  &\geq \int_0^{\infty} t\alpha_D(\Omega)m(\{\phi > t\}) dt \\
  &=\alpha_D(\Omega) \int_0^{\infty}t \sum_{x \in \{\phi > t\}}m(x) dt \\
  &=\alpha_D(\Omega)\sum_{x \in \{\phi > 0\}}m(x) \int_0^{\phi(x)} t dt \\
  &=\frac12\alpha_D(\Omega) \|\phi\|_{l^2(\Omega)}^2,
 \end{align*}
 i.e., $\lambda_1^D(\Omega) \geq \frac14\alpha_D(\Omega)$. This proves the theorem.
\end{proof}

In the following, we estimate the bottom of the spectrum of an infinite weighted graph $G$ by applying the exhaustion methods and \eqref{CheegerInequalityForDirichlet}. 

For an infinite weighted graph $G = (V, E, w, m)$, we define the $k$-th variational eigenvalue of the Laplacian on $G$ as
\begin{equation}\label{KthEigenvalueOfLaplacian}
\lambda_k(G):= \inf_{\substack{H \subseteq l_0(V) \\ \dim H = k}}\max_{0 \neq f \in H}\frac{\mathcal{E}(f, f)}{\|f\|_{l^2(V)}^2}.
\end{equation}
For $k = 1$, $\lambda_1(G)$ is called the bottom of the spectrum of the Laplacian. Let $\mathcal{W} = \{W_i\}_{i=1}^{\infty}$ be an exhaustion of $G$. For any $k \geq 1$, we denote by $\lambda_k^D(W_i)$ the $k$-th Dirichlet eigenvalue on $W_i$. Then by the Rayleigh quotient characterization, for any $i \in \mathbb{N}_+$,
\[
 \lambda_k^D(W_i) \geq \lambda_k^D(W_{i+1}).
\]
One can also show that
\[
 \lambda_k(G) = \lim_{i \to \infty}\lambda_k^D(W_i).
\]
For any $A \subseteq V$ with $|A|<+\infty$, there exists $i \in \mathbb{N}_+$ such that $A \subseteq W_i$. Since
\[
 \mathrm{Cap}_{W_i}(A) = \inf\{\mathcal{E}_{W_i}(f, f): f|_{A} = 1, f|_{\delta W_i} = 0\},
\]
one easily checks that
\[
 \mathrm{Cap}_{W_i}(A) \geq \mathrm{Cap}_{W_{i+1}}(A).
\]
\begin{proof}[Proof of Theorem \ref{CheegerInequalityForBottom}]
 By applying Theorem \ref{CheegerInequalityForDirichlet} and the above argument, we have
 \[
  \frac14\lim_{i \to \infty}\alpha_D(W_i) \leq \lambda_1(G) \leq \lim_{i \to \infty}\alpha_D(W_i). 
 \]
 We only need to prove that $\alpha_D(G) = \lim_{i \to \infty}\alpha_D(W_i)$.
 
 Recall that 
 \[
  \alpha_D(G) = \inf_{\substack{A \subseteq V \\ |A|<+\infty}}\frac{\mathrm{Cap}(A)}{m(A)}.
 \]
 For any $A \subseteq V$ with $|A|<+\infty$, there exist $i \in \mathbb{N}_+$ such that $A \subseteq W_i$. Then we have
 \begin{align*}
  \frac{\mathrm{Cap}(A)}{m(A)} &= \lim_{i \to \infty}\frac{\mathrm{Cap}_{W_i}(A)}{m(A)} \\
  &\geq \limsup_{i \to \infty}\alpha_D(W_i),
 \end{align*}
 and hence,
 \[
  \alpha_D(G) \geq \limsup_{i \to \infty}\alpha_D(W_i).
 \]
 On the other hand, for any $i \in \mathbb{N}_+$, let $A_i \subseteq W_i$ be a finite subset such that
 \[
  \alpha_D(W_i) = \frac{\mathrm{Cap}_{W_i}(A_i)}{m(A_i)}.
 \]
 Then $|A_i|<+\infty$ and
 \begin{align*}
  \alpha_D(W_i) &= \frac{\mathrm{Cap}_{W_i}(A_i)}{m(A_i)} \\
  &\geq \lim_{j \to \infty} \frac{\mathrm{Cap}_{W_j}(A_i)}{m(A_i)} \\
  &= \frac{\mathrm{Cap}(A_i)}{m(A_i)} \geq \alpha_D(G).
 \end{align*}
 It follows that $\liminf_{i \to \infty}\alpha_D(W_i) \geq \alpha_D(G)$. This proves the result.
\end{proof}

As a corollary of Theorem \ref{CheegerInequalityForBottom}, we have the following proposition.
\begin{proposition}\label{ZeroBottomAndRecurrent}
 Let $G = (V, E, w, m)$ be an infinite weighted graph. If $G$ is recurrent, then $\lambda_1(G) = 0$.
\end{proposition}
\begin{proof}
 Note that $G$ is recurrent if and only if we have $\mathrm{Cap}(A) = 0$ for any finite subset $A \subseteq V$, see, e.g., \cite{Woess2000} Theorem 2.12. Thus, if $G$ is recurrent, then $\alpha_D(G) = 0$ and the conclusion follows from Theorem \ref{CheegerInequalityForBottom}.
\end{proof}

At the end of this section, we give an example for the transient case. For $N \geq 1$, let $(\mathbb{Z}^N, E)$ be the $N$-dimensional lattice graph with unit edge weights, where
 \[
  E = \left\{\{x, y\}: x, y \in \mathbb{Z}^N, \sum_{i=1}^{N}|x_i - y_i| = 1\right\}.
 \]
\begin{example}\label{TransientCase}
 For $N \geq 3$, consider $G = (\mathbb{Z}^N, E)$. Suppose that the volume of $\mathbb{Z}^N$ is finite, i.e., $m(\mathbb{Z}^N) < +\infty$. Then according to \cite[Theorem 1.6]{HKSW2023}, there exists a constant $C>0$ such that
 \[
  \lambda_1(G) \geq C\left(\frac{1}{m(\mathbb{Z}^N)}\right)^{1-\frac{2}{p}} = C\left(\frac{1}{m(\mathbb{Z}^N)}\right)^{\frac{2}{N}},
 \]
 where $p = \frac{2N}{N-2}$. By applying Theorem \ref{CheegerInequalityForBottom}, we can get a new estimate. Note that for any finite $A \subset \mathbb{Z}^N$, 
 \[
  \frac{\mathrm{Cap}(A)}{m(A)} \geq \frac{\mathrm{Cap}(\{p\})}{m(\mathbb{Z}^n)}, \quad \forall \ p \in A.
 \]
 Since $N \geq 3$, for any $p \in \mathbb{Z}^N$, by \cite[Theorem 2.12]{Woess2000},
 \[
  \mathrm{Cap}(\{p\}) = C(N).
 \]
 Then
 \[
  \alpha_D(G) \geq \frac{C(N)}{m(\mathbb{Z}^N)}.
 \]
 On the other hand, 
 \[
  \alpha_D(G) \leq \inf_{p \in \mathbb{Z}^N} \frac{\mathrm{Cap}(\{p\})}{m(\{p\})}.
 \]
 Thus, by Theorem \ref{CheegerInequalityForBottom}, we have
 \[
  \frac{C(N)}{m(\mathbb{Z}^N)} \leq \lambda_1(G) \leq \inf_{p \in \mathbb{Z}^N} \frac{\mathrm{Cap}(\{p\})}{m(\{p\})}. 
 \]
 Note that the lower bound $\frac{1}{m(\mathbb{Z}^N)}$ is better than the estimate in \cite[Theorem 1.6]{HKSW2023}, $\left(\frac{1}{m(\mathbb{Z}^N)}\right)^{\frac{2}{N}}$, if $m(\mathbb{Z}^N) < 1$.
\end{example}

\section{First Neumann Eigenvalue}

In this section, we estimate the first non-trivial Neumann eigenvalue. Let $G = (V, E, w, m)$ be a weighted graph, $\Omega \subseteq V$ be a finite subset satisfying $|\Omega| \geq 2$. Consider the graph $G_{\Omega}$. The Neumann problem on $\Omega$ is equivalent to the following eigenvalue problem
\begin{equation*}
 \begin{cases}
  \widetilde{\mathcal{L}}f(x) = \lambda f(x), \quad &x \in \Omega, \\
  \widetilde{\mathcal{L}}f(x) = 0, &x \in \delta \Omega,
 \end{cases}
\end{equation*}
where $\widetilde{\mathcal{L}}$ is the Laplacian on $G_{\Omega}$. This is a special case of the Steklov problem studied in \cite{HM2020}.

Let us first consider the eigenvalue problem of the Laplace operator on the graph $G_{\Omega}$, i.e.,
\[
 \widetilde{\mathcal{L}}f(x) = \lambda f(x), \quad \forall \ x \in \overline{\Omega}.
\]
We denote by $\lambda_1^{\widetilde{\mathcal{L}}}$ the first non-trivial eigenvalue of $\widetilde{\mathcal{L}}$. By Rayleigh quotient characterization,
\[
 \lambda_1^{\widetilde{\mathcal{L}}} = \inf_{\substack{f \in \mathbb{R}^{\overline{\Omega}} \\ f \not\equiv 0}}\frac{\mathcal{E}_{\Omega}(f, f)}{\min\limits_{c \in \mathbb{R}}\|f-c\|_{l^2(\overline{\Omega})}^2}.
\]
Then we have the following estimate, see also \cite{SS2019}.
\begin{theorem}
 The first non-trivial eigenvalue $\lambda_1^{\widetilde{\mathcal{L}}}$ satisfies
 \begin{equation}\label{CheegerEstimate}
  \frac18\widetilde{\alpha}_{\widetilde{\mathcal{L}}} \leq \lambda_1^{\widetilde{\mathcal{L}}} \leq 2\widetilde{\alpha}_{\widetilde{\mathcal{L}}},
 \end{equation}
 where $\widetilde{\alpha}_{\widetilde{\mathcal{L}}}:= \inf\limits_{A, B \subseteq \overline{\Omega}}\frac{\mathrm{Cap}_{\Omega}(A, B)}{m(A) \wedge m(B)}$.
\end{theorem}
\begin{proof}
 We first show that $\lambda_1^{\widetilde{\mathcal{L}}} \leq 2\widetilde{\alpha}_{\widetilde{\mathcal{L}}}$. Let $A, B \subseteq \overline{\Omega}$ be subsets such that 
 \[
  \widetilde{\alpha}_{\widetilde{\mathcal{L}}} = \frac{\mathrm{Cap}_{\Omega}(A, B)}{m(A) \wedge m(B)}.
 \]
 Then there exists a function $f$ satisfying $f|_A = 1$ and $f|_B = 0$, such that
 \[
  \mathcal{E}_{\Omega}(f, f) = \mathrm{Cap}_{\Omega}(A, B) = \widetilde{\alpha}_{\widetilde{\mathcal{L}}}\left(m(A) \wedge m(B)\right).
 \]
 Then for any $c \in \mathbb{R}$,
 \begin{align*}
  \|f-c\|_{l^2(\overline{\Omega})}^2 &\geq \sum_{x \in A}|f(x) - c|^2m(x) + \sum_{x \in B}|f(x) - c|^2m(x) \\
  &=(1-c)^2m(A) + c^2m(B) \\
  &\geq (m(A) \wedge m(B))(c^2 + (1-c)^2) \\
  &\geq \frac12(m(A) \wedge m(B)).
 \end{align*}
 Thus,
 \[
  \min_{c \in \mathbb{R}}\|f-c\|_{l^2(\overline{\Omega})}^2 \geq \frac12(m(A) \wedge m(B)).
 \]
 It follows that
 \[
  2\widetilde{\alpha}_{\widetilde{\mathcal{L}}} = \frac{2\mathcal{E}_{\Omega}(f, f)}{m(A) \wedge m(B)} \geq \frac{\mathcal{E}_{\Omega}(f, f)}{\min_{c \in \mathbb{R}}\|f-c\|_{l^2(\overline{\Omega})}^2} \geq \lambda_1^{\widetilde{\mathcal{\mathcal{L}}}}.
 \]
 To show that $\frac{1}{8}\widetilde{\alpha}_{\widetilde{\mathcal{L}}} \leq \lambda_1^{\widetilde{\mathcal{L}}}$, let $\phi \in \mathbb{R}^{\overline{\Omega}}$ be a first eigenfunction. Then there exists $c_1 \in \mathbb{R}$ such that $m\left(\{\phi + c_1 \leq 0\}\right) \geq \frac12m(\overline{\Omega})$ and $m\left(\{\phi + c_1 \geq 0\}\right) \geq \frac12m(\overline{\Omega})$. Denote $f = \phi + c_1$. Since $\|f\|_{l^2(\overline{\Omega})}^2 = \|f^+\|_{l^2(\overline{\Omega})}^2 + \|f^-\|_{l^2(\overline{\Omega})}^2$, where $f^+ = \max\{f, 0\}$ and $f^-=\max\{-f,0\}$, w.l.o.g., we assume that $\|f^+\|_{l^2(\overline{\Omega})}^2 \geq \frac12\|f\|_{l^2(\overline{\Omega})}^2$. Then by Lemma \ref{CoareaFormula},
 \begin{align*}
  2\lambda_1^{\widetilde{\mathcal{L}}}\|f\|_{l^2(\overline{\Omega})}^2 &\geq 2\mathcal{E}_{\Omega}(f, f) \geq \int_0^{\infty} t\mathrm{Cap}_{\Omega}(\{f>t\}, \{f\leq 0\})dt \\
  &\geq \int_0^{\infty} t\widetilde{\alpha}_{\widetilde{\mathcal{L}}}\left(m\left(\{f > t\}\right) \wedge m\left(\{f \leq 0\}\right)\right)dt \\
  &= \int_0^{\infty}t\widetilde{\alpha}_{\widetilde{\mathcal{L}}}m\left(\{f>t\}\right)dt \\
  &=\widetilde{\alpha}_{\widetilde{\mathcal{L}}}\int_0^{\infty}t\sum_{x \in \{f > t\}}m(x)dt \\
  &=\widetilde{\alpha}_{\widetilde{\mathcal{L}}}\sum_{x \in \{f>0\}}m(x)\int_0^{f(x)}tdt \\
  &=\frac12\widetilde{\alpha}_{\widetilde{\mathcal{L}}}\|f^+\|_{l^2(\overline{\Omega})}^2 \geq \frac14\widetilde{\alpha}_{\widetilde{\mathcal{L}}}\|f\|_{l^2(\overline{\Omega})}^2,
 \end{align*}
 i.e., $\lambda_1^{\widetilde{\mathcal{L}}} \geq \frac18\widetilde{\alpha}_{\widetilde{\mathcal{L}}}$. This finishes the proof.
\end{proof}

Now, we construct a sequence of new finite graphs $\{G^{(k)}\}_{k=1}^{\infty}$ with $G^{(1)} = G_{\Omega}$ and $G^{(k)}=(\overline{\Omega}, E(\Omega, \overline{\Omega}), w, m^{(k)})$, where $m^{(k)}|_{\Omega} = m$ and $m^{(k)}|_{\delta \Omega} = \frac{m}{k}$. Denote by $\lambda_1^{(k)}$ the first non-trivial eigenvalue of the Laplacian on the graph $G^{(k)}$. Then by \eqref{CheegerEstimate}, we have
\[
 \frac18\widetilde{\alpha}_{\widetilde{\mathcal{L}}}^{(k)} \leq \lambda_1^{(k)} \leq 2\widetilde{\alpha}_{\widetilde{\mathcal{L}}}^{(k)},
\]
where
\[
 \widetilde{\alpha}_{\widetilde{\mathcal{L}}}^{(k)} = \inf_{A, B \subseteq \overline{\Omega}}\frac{\mathrm{Cap}_{\Omega}(A, B)}{m^{(k)}(A) \wedge m^{(k)}(B)}.
\]
On the other hand, by applying \cite[Proposition 3]{HM2020}, we have
\[
 \lim_{k \to \infty} \lambda_1^{(k)} = \lambda_1^{N}(\Omega).
\]
In order to prove Theorem \ref{CheegerInequalityForNeumann}, we first show that $\lim_{k \to \infty}\widetilde{\alpha}_{\widetilde{\mathcal{L}}}^{(k)}$ is well-defined.

\begin{lemma}\label{MonotoneOfCapacity}
 For any $\Omega \subseteq V$ with $|\Omega| \geq 2$ and any $i \in \mathbb{N}_+$, we have
 \[
  0 < \widetilde{\alpha}_{\widetilde{\mathcal{L}}}^{(i)} \leq \widetilde{\alpha}_{\widetilde{\mathcal{L}}}^{(i+1)} \leq C,
 \]
 where $C$ is a constant depending only on $G_{\Omega}$.
\end{lemma}
\begin{proof}
 Let $A, B \subseteq \overline{\Omega}$ be finite subsets such that
 \[
  \widetilde{\alpha}_{\widetilde{\mathcal{L}}}^{(i+1)} = \frac{\mathrm{Cap}_{\Omega}(A, B)}{m^{(i+1)}(A) \wedge m^{(i+1)}(B)}.
 \] 
 Since $m^{(i+1)} \leq m^{(i)}$, we have
 \[
  m^{(i)}(A) \wedge m^{(i)}(B) \geq m^{(i+1)}(A) \wedge m^{(i+1)}(B).
 \]
 It follows that
 \[
  \widetilde{\alpha}_{\widetilde{\mathcal{L}}}^{(i+1)} = \frac{\mathrm{Cap}_{\Omega}(A, B)}{m^{(i+1)}(A) \wedge m^{(i+1)}(B)} \geq \frac{\mathrm{Cap}_{\Omega}(A, B)}{m^{(i)}(A) \wedge m^{(i)}(B)} \geq \widetilde{\alpha}_{\widetilde{\mathcal{L}}}^{(i)}.
 \]
 Moreover, 
 \[
  \widetilde{\alpha}_{\widetilde{\mathcal{L}}}^{(i)} \geq \widetilde{\alpha}_{\widetilde{\mathcal{L}}} > 0.
 \]
 On the other hand, since $|\Omega| \geq 2$, we can take $x_1, x_2 \in \Omega$. Then for any $k \in \mathbb{N}_+$,
 \[
  \widetilde{\alpha}_{\widetilde{\mathcal{L}}}^{(k)} \leq \frac{\mathrm{Cap}_{\Omega}(x_1, x_2)}{m^{(k)}(x_1) \wedge m^{(k)}(x_2)} = \frac{\mathrm{Cap}_{\Omega}(x_1, x_2)}{m(x_1) \wedge m(x_2)} =: C.
 \]
 This proves the lemma.
\end{proof}
\begin{proof}[Proof of Theorem \ref{CheegerInequalityForNeumann}]
 By Lemma \ref{MonotoneOfCapacity}, we know that $\lim_{k \to \infty}\widetilde{\alpha}_{\widetilde{\mathcal{L}}}^{(k)}$ is well-defined. Hence,
 \[
  \frac18\lim_{k \to \infty}\widetilde{\alpha}_{\widetilde{\mathcal{L}}}^{(k)} \leq \lambda_1^N(\Omega) \leq 2\lim_{k \to \infty}\widetilde{\alpha}_{\widetilde{\mathcal{L}}}^{(k)}.
 \]
 We only need to prove that 
 \[
  \alpha_N(\Omega) = \lim_{k \to \infty}\widetilde{\alpha}_{\widetilde{\mathcal{L}}}^{(k)}.
 \]
 Recall that 
 \[
  \alpha_N(\Omega) = \inf_{A, B \subseteq \Omega}\frac{\mathrm{Cap}_{\Omega}(A, B)}{m(A) \wedge m(B)}.
 \]
 For any $A, B \subseteq \Omega$ and $k \in \mathbb{N}_+$,
 \[
  \frac{\mathrm{Cap}_{\Omega}(A, B)}{m(A) \wedge m(B)} = \frac{\mathrm{Cap}_{\Omega}(A, B)}{m^{(k)}(A) \wedge m^{(k)}(B)} \geq \widetilde{\alpha}_{\widetilde{\mathcal{L}}}^{(k)},
 \]
 and thus,
 \[
  \alpha_N(\Omega) \geq \lim_{k \to \infty}\widetilde{\alpha}_{\widetilde{\mathcal{L}}}^{(k)}.
 \]
 On the other hand, let $A^{(k)}, B^{(k)} \subseteq \overline{\Omega}$ be finite subsets such that
 \[
  \widetilde{\alpha}_{\widetilde{\mathcal{L}}}^{(k)} = \frac{\mathrm{Cap}_{\Omega}(A^{(k)}, B^{(k)})}{m^{(k)}(A^{(k)}) \wedge m^{(k)}(B^{(k)})}.
 \]
 Since $\overline{\Omega}$ is finite, we can choose a subsequence of $A^{(k)}, B^{(k)}$, still denoted by $A^{(k)}$ and $B^{(k)}$, such that
 \[
  A^{(k)} = A \subseteq \overline{\Omega}, \quad B^{(k)} = B \subseteq \overline{\Omega}.
 \]
 Hence,
 \[
  \widetilde{\alpha}_{\widetilde{\mathcal{L}}}^{(k)} = \frac{\mathrm{Cap}_{\Omega}(A, B)}{m^{(k)}(A) \wedge m^{(k)}(B)} \geq \frac{\mathrm{Cap}_{\Omega}(A \cap \Omega, B \cap \Omega)}{m^{(k)}(A) \wedge m^{(k)}(B)}.
 \]
 It follows that 
 \begin{align*}
  \lim_{k \to \infty}\widetilde{\alpha}_{\widetilde{\mathcal{L}}}^{(k)} &\geq \lim_{k \to \infty}\frac{\mathrm{Cap}_{\Omega}(A \cap \Omega, B \cap \Omega)}{m^{(k)}(A) \wedge m^{(k)}(B)} \\
  &=\frac{\mathrm{Cap}_{\Omega}(A \cap \Omega, B \cap \Omega)}{m(A \cap \Omega) \wedge m(B \cap \Omega)}  \geq \alpha_N(\Omega).
 \end{align*}
 This finishes the proof.
\end{proof}

\section{First Steklov Eigenvalue}

Let $G = (V, E, w, m)$ be a weighted graph, $\Omega \subseteq V$ be a finite subset whose vertex boundary $\delta \Omega$ satisfies $|\delta \Omega| \geq 2$. Let $\sigma_1(\Omega)$ be the first non-trivial Steklov eigenvalue on $\Omega$. By the Rayleigh quotient characterization, 
\[
 \sigma_1(\Omega) = \inf_{\substack{f \in \mathbb{R}^{\delta \Omega}, f \not\equiv 0 \\ \langle f, 1 \rangle_{\delta \Omega} = 0}}\frac{\langle \Lambda_{\Omega}f, f \rangle_{\delta \Omega}}{\|f\|_{l^2(\delta \Omega)}^2} = \inf_{\substack{f \in \mathbb{R}^{\delta \Omega}, f \not\equiv 0 \\ \langle f, 1 \rangle_{\delta \Omega} = 0}}\frac{\mathcal{E}_{\Omega}(u_f, u_f)}{\|f\|_{l^2(\delta \Omega)}^2},
\]
where $u_f$ is the harmonic extension of $f$ to $\Omega$. In order to prove Theorem \ref{CheegerInequalityForSteklov}, we construct a sequence of new finite graphs $\{G^{(k)}\}_{k=1}^{\infty}$ with $G^{(k)} = (\overline{\Omega}, E(\Omega, \overline{\Omega}), w, m^{(k)})$, where $m^{(k)}$ satisfies $m^{(k)}|_{\delta\Omega} = m$ and $m^{(k)}|_{\Omega} = \frac{m}{k}$. We denote by $\lambda_1^{(k)}$ the first non-trivial eigenvalue of the Laplacian on $G^{(k)}$. Then by \cite[Proposition 3]{HM2020}, we have
\begin{equation}\label{LimitEigenvalue}
 \lim\limits_{k \to \infty} \lambda_1^{(k)} = \sigma_1(\Omega).
\end{equation}
Moreover, by \eqref{CheegerEstimate}, for any $k \in \mathbb{N}_+$, we have
\begin{equation}\label{NeedForTheorem}
 \frac18\widetilde{\alpha}_{\widetilde{\mathcal{L}}}^{(k)} \leq \lambda_1^{(k)} \leq 2\widetilde{\alpha}_{\widetilde{\mathcal{L}}}^{(k)},
\end{equation}
where 
\[
 \widetilde{\alpha}_{\widetilde{\mathcal{L}}}^{(k)} = \inf_{A, B \subseteq \overline{\Omega}}\frac{\mathrm{Cap}_{\Omega}(A, B)}{m^{(k)}(A) \wedge m^{(k)}(B)}.
\]
A proof similar to the proof of Lemma \ref{MonotoneOfCapacity} shows that $\lim_{k\to\infty}\widetilde{\alpha}_{\widetilde{\mathcal{L}}}^{(k)}$ is well-defined. 
\begin{proof}[Proof of Theorem \ref{CheegerInequalityForSteklov}]
 The proof is the same as that of Theorem \ref{CheegerInequalityForNeumann} by switching the roles of $\Omega$ and $\delta \Omega$. Due to \eqref{LimitEigenvalue} and \eqref{NeedForTheorem}, we have
 \[
  \frac18\lim_{k \to \infty}\widetilde{\alpha}_{\widetilde{\mathcal{L}}}^{(k)} \leq \sigma_1(\Omega) \leq 2\lim_{k \to \infty}\widetilde{\alpha}_{\widetilde{\mathcal{L}}}^{(k)}.
 \]
 A same argument as in Theorem \ref{CheegerInequalityForNeumann} yields 
 \[
  \lim_{k \to \infty}\widetilde{\alpha}_{\widetilde{\mathcal{L}}}^{(k)} = \alpha_S(\Omega).
 \]
 This proves the result.
\end{proof}

\begin{example}\label{Line}
 Let $G = (\mathbb{Z}, E)$ be $1$-dimensional lattice graph with unit vertex weights, $\Omega = \{1, 2, \cdots, n-1\}$.
 \begin{figure}[H]
  \centering
   \begin{tikzpicture}
    \filldraw[draw = black, fill = black] (0,0) circle(0.05) (1,0) circle(0.05) (2,0) circle(0.05) (3.5,0) circle(0.05) (4.5,0) circle(0.05) (5.5,0)circle(0.05);
    \draw (0,0) -- (2.2,0); \draw (3.3,0) -- (5.5,0);
    \draw[densely dashed] (2.3,0) -- (3.2,0);
    \node at (0,-0.3) {$0$}; \node at (1,-0.3) {$1$}; \node at (2,-0.3) {$2$}; \node at (3.5,-0.3) {$n-2$}; \node at (4.5,-0.3) {$n-1$}; \node at (5.5, -0.3) {$n$};
    \draw[densely dashed] (0.5,0.5) -- (5,0.5) -- (5,-0.5) -- (0.5, -0.5) -- (0.5,0.5);
    \node at (3, 0.3) {$\Omega$};
   \end{tikzpicture}
  \caption{}
 \end{figure}
 Choose $A = \{0\}$, $B = \{n\}$. Then $\mathrm{Cap}_{\Omega}(A, B) = \frac{1}{n}$. Thus, by Theorem \ref{CheegerInequalityForSteklov}, we have
 \[
  \frac{1}{8n} \leq \sigma_1(\Omega) \leq \frac{2}{n}.
 \]
 In fact, by direct calculation, we have $\sigma_1(\Omega) = \frac{2}{n}$. This example shows that the upper bound estimate in Theorem \ref{CheegerInequalityForSteklov} is sharp. 
\end{example}

\begin{example}\label{FiniteTree}
 Consider the following finite tree $T_3$ with unit edge weights and vertex weights, $\Omega = \{x_1, x_2, x_3, x_4\}$.
 \begin{figure}[H]
 \centering
  \begin{tikzpicture}
   \filldraw[draw = black, fill = black] (0,0) circle(0.05) (1,0) circle(0.05) (1.8,0.5) circle(0.05) (1.8,-0.5) circle(0.05) (2.75,0.2) circle(0.05) (2.75,0.8) circle(0.05);
   \filldraw[draw = black, fill = black] (2.75,-0.2) circle(0.05) (2.75,-0.8) circle(0.05) (-0.8,0.5) circle(0.05) (-0.8,-0.5) circle(0.05);
   \draw (-0.8,0.5) -- (0,0) -- (1,0) -- (1.8,0.5) -- (2.75,0.8);
   \draw (-0.8,-0.5) -- (0,0); \draw (1.8,0.5) -- (2.75,0.2); \draw(1,0) -- (1.8,-0.5) -- (2.75,-0.8); \draw (1.8,-0.5) -- (2.75,-0.2);
   \node at (1,0.3) {$x_1$}; \node at (0,0.3) {$x_2$}; \node at (1.6,0.7) {$x_3$}; \node at (1.6,-0.7) {$x_4$}; \node at (-1.1, 0.5) {$x_5$};
   \node at (-1.1,-0.5) {$x_6$}; \node at (3.05,0.8) {$x_7$}; \node at (3.05,0.2) {$x_8$}; \node at (3.05,-0.2) {$x_9$}; \node at (3.1,-0.8) {$x_{10}$};
  \end{tikzpicture}
  \caption{}
\end{figure}
 Choose $A = \{x_5, x_6\}$ and $B = \{x_7, x_8\}$. Then one gets $\mathrm{Cap}_{\Omega}(A, B) = \frac{1}{3}$. It follows that $\alpha_S(\Omega) = \frac16$. Thus,
 \[
  \frac{1}{48} \leq \sigma_1(\Omega) \leq \frac13.
 \]
 One can actually calculate that $\sigma_1(\Omega) = \frac13$. Our upper bound estimate is sharp on $T_3$. 
\end{example}

We give a necessary and sufficient condition for the equality case, $\sigma_1(\Omega) = 2\alpha_S(\Omega)$.
\begin{proposition}\label{EqualityCase}
 Let $G$ be a weighted graph, $\Omega \subseteq V$ be a finite subset with $|\delta \Omega| \geq 2$. Then $\sigma_1(\Omega) = 2\alpha_S(\Omega)$ if and only if there exists an eigenfunction $f$ corresponding to $\sigma_1(\Omega)$ such that for any $x \in \delta \Omega$, $f(x) \in \{-1, 1, 0\}$.
\end{proposition}
\begin{proof}
  For the ``if'' part. One easily checks that for any $X, Y \subseteq \overline{\Omega}$, 
 \[
  \mathrm{Cap}_{\Omega}(X, Y) = \frac14\inf\{\mathcal{E}_{\Omega}(g, g): g|_X = 1, g|_Y = -1\}.
 \]
 Let $f$ be such an eigenfunction, set
 \[
  A = \{x \in \delta \Omega: f(x) = 1\},
 \]
 and
 \[
  B = \{x \in \delta \Omega: f(x) = -1\}.
 \]
 Then we have
 \[
  \mathrm{Cap}_{\Omega}(A, B) \leq \frac14\mathcal{E}_{\Omega}(f, f).
 \]
 On the other hand, $\sum_{x \in \delta \Omega}f(x)m(x) = 0$ implies that $m(A) = m(B)$. Thus,
 \begin{align*}
  \alpha_S(\Omega) &\leq \frac{\mathrm{Cap}_{\Omega}(A, B)}{m(A) \wedge m(B)} \leq \frac{\frac14\mathcal{E}_{\Omega}(f, f)}{m(A) \wedge m(B)} \\
  &=\frac{\frac14\mathcal{E}_{\Omega}(f, f)}{\frac12(m(A)+m(B))} = \frac12\sigma_1(\Omega).
 \end{align*}
 
 For the ``only if'' part. Let $A, B \subseteq \delta \Omega$ be subsets such that 
 \[
  \alpha_S(\Omega) = \frac{\mathrm{Cap}_{\Omega}(A, B)}{m(A) \wedge m(B)}.
 \]
 Then there exists a function $g$ satisfying $g|_A = 1$ and $g|_B = 0$, such that $\mathcal{E}_{\Omega}(g, g) = \mathrm{Cap}_{\Omega}(A, B)$. Same as the proof of Theorem \ref{CheegerEstimate}, for any $c \in \mathbb{R}$, 
 \begin{align*}
  \|g-c\|_{l^2(\delta \Omega)}^2 &\geq \sum_{x \in A}(1-c)^2m(x) + \sum_{x \in B}c^2m(x) \\
  &\geq (m(A) \wedge m(B))((1-c)^2 + c^2) \\
  &\geq \frac12(m(A) \wedge m(B)).
 \end{align*} 
 Thus, 
 \[
  \min_{c \in \mathbb{R}}\|g-c\|_{l^2(\delta \Omega)}^2 \geq \frac12(m(A) \wedge m(B)).
 \]
 By the Rayleigh quotient characterization and $\sigma_1(\Omega) = 2\alpha_S(\Omega)$, we have
 \[
  \sigma_1(\Omega) \leq \frac{\mathcal{E}_{\Omega}(g, g)}{\min_{c \in \mathbb{R}}\|g-c\|_{l^2(\delta \Omega)}^2} \leq 2\alpha_S(\Omega) = \sigma_1(\Omega).
 \]
 This implies that $\min_{c \in \mathbb{R}}\|g-c\|_{l^2(\delta \Omega)}^2 = \|g- \frac12\|_{l^2(\delta \Omega)}^2$ and all inequalities above are equalities. It follows that $g-\frac12$ is a first eigenfunction and for any $x \in \delta \Omega$, $g-\frac12 \in \{-\frac12, 0, \frac12\}$.
 This proves the result.
\end{proof}

In the following, we estimate the bottom of the spectrum of the DtN operator on an infinite subgraph. We first recall the definition of the DtN operator on an infinite subgraph of a graph introduced in \cite{HHW2022}. 

Let $U \subseteq V$ be an infinite subset. We consider the graph structure $G_U$. Note that edges between vertices in $\delta U$ are removed. For any $f \in l_0(\delta U)$, we write $f = f^+ - f^-$, where $f^+:= \max\{f, 0\}$ and $f^-:= \max\{-f, 0\}$. For any $\mathcal{W}=\{W_i\}_{i=1}^{\infty} \uparrow \overline{U}$, w.l.o.g., we assume that $W_1 \cap \delta U \neq \varnothing$. Let $u^{W_i}_{f^+}$ be the solution of the following equation:
\begin{equation*}
 \begin{cases}
  \Delta u^{W_i}_{f^+}(x) = 0, \quad &x \in W_i \cap U, \\
  u^{W_i}_{f^+}(x) = f^+(x), &x \in W_i \cap \delta U,\\
  u^{W_i}_{f^+}(x) = 0, &x \in \overline{U} \setminus W_i,
 \end{cases}
\end{equation*}
and $u^{W_i}_{f^-}$ be the solution of the equation with $f^+$ replaced by $f^-$. Set 
\[
 u^{W_i}_f:= u^{W_i}_{f^+} - u^{W_i}_{f^-},
\]
and define 
\[
 u_f:= \lim_{i \to \infty}u^{W_i}_f.
\]
Note that $u_f$ exists by the maximum principle. For any $f \in l_0(\delta U)$, we define
\[
 \Lambda(f):= \frac{\partial u_f}{\partial n},
\]
where $\frac{\partial}{\partial n}$ is taken w.r.t. $U$. One can check that $\Lambda$ is a symmetric linear operator on $l_0(\delta U)$, see, e.g. \cite{HHW2022}. We define a form $Q: l_0(\delta U) \to [0, \infty)$ via
 \[
  Q(f) = \langle \Lambda(f), f \rangle_{\delta U} = \sum_{\{x, y\} \in E(U, \overline{U})}w(x, y)(u_f(y) - u_f(x))^2,
 \]
Denote by $D$ the closure of $l_0(\delta U)$ w.r.t. the norm $\|\cdot\| = \sqrt{\|\cdot\|^2_{l^2(\delta U)}+Q(\cdot)}$. Define 
\begin{equation*}
 \widetilde{Q}(f) = 
  \begin{cases}
   \sum_{\{x, y\} \in E(U, \overline{U})}w(x, y)(u_f(y) - u_f(x))^2, \quad &f \in D, \\
   +\infty, &f \notin D.
  \end{cases}
\end{equation*}
One easily sees that $\widetilde{Q}$ is lower continuous, i.e., $\widetilde{Q}$ is closed. On the other hand, for any $f \in \l_0(\delta U)$ and any map $T: \mathbb{R} \to \mathbb{R}$ satisfying $T(0) = 0$ and $|T(x)-T(y)| \leq |x-y|$, 
\begin{align*}
 \widetilde{Q}(T\circ f) &= \sum_{\{x, y\} \in E(U, \overline{U})}w(x, y)(u_{T\circ f}(y) - u_{T\circ f}(x))^2 \\
 &\leq \sum_{\{x, y\} \in E(U, \overline{U})}w(x, y)(T\circ u_f(y) - T \circ u_f(x))^2 \\
 &\leq \sum_{\{x, y\} \in E(U, \overline{U})}w(x, y)(u_f(y)-u_f(x))^2 \\
 &=\widetilde{Q}(f).
\end{align*}
It follows that $\widetilde{Q}(T \circ f) \leq \widetilde{Q}(f)$ for any $f \in D$. Thus, $\widetilde{Q}$ is a Dirichlet form. According to \cite[Theorem 1.2.1]{EBDavies1990}, there exists a unique self-adjoint operator, still denoted by $\Lambda$, such that
\[
 D = \text{Domain of definition of } \Lambda^{\frac12}
\]
and for $f \in D$,
\[
 \widetilde{Q}(f) = \langle \Lambda^{\frac12}f, \Lambda^{\frac12}f \rangle_{\delta U}.
\]
We call $\Lambda$ the DtN operator on $U$.

This defines the DtN operator on infinite subsets with arbitrary vertex weights. Note that the original definition in \cite{HHW2022} assumes some conditions on vertex weights to guarantee the boundedness of the operator. 

By the standard spectral theory, see e.g., \cite{EBDavis1995}, the $k$-th variational eigenvalue of the DtN operator $\Lambda$ on $U$ is defined as
\begin{equation}\label{DefinitionOfSigmaK}
 \sigma_k(U):= \inf_{\substack{H \subseteq l_0(\delta U) \\ \dim H = k}}\sup_{0 \neq f \in H}\frac{\langle \Lambda(f), f \rangle_{\delta U}}{\langle f, f \rangle_{\delta U}}.
\end{equation}

For $k =1$, $\sigma_1(U)$ is called the bottom of the spectral of $\Lambda$. Our aim is to estimate $\sigma_1(U)$. For this purpose, let $W \subseteq \overline{U}$ be a finite subset satisfying $W \cap \delta U \neq \varnothing$. We first estimate the eigenvalue of the DtN operator on $W$, denoted by $\Lambda_{W}$, which is defined as
\begin{align*}
 \Lambda_{W}: \mathbb{R}^{W \cap \delta U} &\to \mathbb{R}^{W \cap \delta U}, \\
 f &\mapsto \Lambda_{W}(f):= \frac{\partial u_f^{W}}{\partial n},
\end{align*}
where $\frac{\partial}{\partial n}$ is taken w.r.t. $U$. We denote by $\sigma_k^D(W)$, $1 \leq k \leq |W \cap \delta U|$, the $k$-th eigenvalue of the operator $\Lambda_{W}$. Note that $\sigma_k^D(W)$ can be characterized as 
\begin{equation}\label{DefinitionOfSigmaKOnWi}
 \sigma_k^D(W):= \min_{\substack{H \subseteq \mathbb{R}^{W \cap \delta U} \\ \dim H = k}}\max_{0 \neq f \in H}\frac{\langle \Lambda_{W}(f), f \rangle_{W \cap \delta U}}{\langle f, f \rangle_{W \cap \delta U}}.
\end{equation}
\begin{lemma}\label{CheegerInequalityForDS}
 For any finite $W \subseteq \overline{U}$ with $W \cap \delta U \neq \varnothing$, set 
 \[
  \alpha_{DS}^U(W) = \inf_{A \subseteq W \cap \delta U}\frac{\mathrm{Cap}_{U}(A, \delta_UW)}{m(A)}.
 \]
 Then
 \[
  \frac14\alpha_{DS}^U(W) \leq \sigma_1^D(W) \leq \alpha_{DS}^U(W).
 \]
\end{lemma}
\begin{proof}
 We construct a sequence of finite graphs $\{G^{(k)}\}_{k=1}^{\infty}$ where $$G^{(k)} = (W \cup \delta_U W, E(W, W \cup \delta_U W), w, m^{(k)})$$ with $m^{(k)}|_{W \cap \delta U} = m$ and $m^{(k)}|_{W \cap U} = \frac{m}{k}$. Recall that $\delta_UW$ is the vertex boundary of $W$ in $G_U$. Consider the following Dirichlet problem
 \begin{align*}
 \begin{cases}
  -\Delta^{(k)} f(x) = \lambda f(x), \quad &x \in W, \\
  f(x) = 0, &x \in \delta_U W,
 \end{cases}
 \end{align*}
 where $\Delta^{(k)}$ is the Laplacian on $G^{(k)}$. Denote by $\lambda_1^{(k)}(W)$ the first eigenvalue of the above Dirichlet problem. The proof of the Proposition 3 in \cite{HM2020} can easily be adapted to show that 
 \[
  \lim_{k \to \infty}\lambda^{(k)}_1(W) = \sigma_1^D(W).
 \]
 By applying Theorem \ref{CheegerInequalityForDirichlet}, we have
 \[
  \frac14\alpha_D^{(k)}(W) \leq \lambda^{(k)}_1(W) \leq \alpha_D^{(k)}(W),
 \]
 where 
 \[
  \alpha_D^{(k)}(W) = \inf_{A \subseteq W}\frac{\mathrm{Cap}_{U}(A, \delta_UW)}{m^{(k)}(A)}.
 \]
 Following the proof of Theorem \ref{CheegerInequalityForNeumann}, we get $$\lim_{k \to \infty}\alpha_D^{(k)}(W) = \alpha_{DS}^U(W).$$
 This proves the lemma.
\end{proof}

For any $A \subseteq \delta U$ with $|A| < +\infty$, set
\[
 \alpha_S(U) = \inf_{\substack{A \subseteq \delta U \\ |A|<+\infty}}\frac{\mathrm{Cap}^U(A)}{m(A)}.
\]
For an exhaustion $\mathcal{W}=\{W_i\}_{i=1}^{\infty}$ of $\overline{U}$, there exists $i \in \mathbb{N}_+$ such that $A \subseteq W_i \cap \delta U$. One easily checks that 
\[
 \frac{\mathrm{Cap}_U(A, \delta_U W_i)}{m(A)} \geq \frac{\mathrm{Cap}_U(A, \delta_U W_{i+1})}{m(A)},
\]
which means that $\alpha_{DS}^U(W_i) \geq \alpha_{DS}^U(W_{i+1})$. Hence, $\lim_{i \to \infty}\alpha_{DS}^U(W_i)$ exists. 
\begin{lemma}\label{LimitDSConstant}
Let $G$ be an infinite weighted graph, $U \subseteq V$ be an infinite subset. Let $\mathcal{W} = \{W_i\}_{i=1}^{\infty} \uparrow \overline{U}$ be an exhaustion, then
 $$\lim_{i \to \infty}\alpha_{DS}^U(W_i) = \alpha_S(U).$$
\end{lemma}
\begin{proof}
 For any $A \subseteq \delta U$ with $|A|<+\infty$, there exists $i \in \mathbb{N}_+$ such that $A \subseteq W_i \cap \delta U$. Thus, 
 \[
  \frac{\mathrm{Cap}^U(A)}{m(A)} = \lim_{i \to \infty}\frac{\mathrm{Cap}_{U}(A, \delta_U W_i)}{m(A)} \geq \lim_{i \to \infty}\alpha_{DS}^U(W_i).
 \]
 It follows that $\alpha_S(U) \geq \lim_{i \to \infty}\alpha_{DS}^U(W_i)$. On the other hand, let $A \subseteq W_i \cap \delta U$ be a subset such that
 \[
  \alpha_{DS}^U(W_i) = \frac{\mathrm{Cap}_{U}(A, \delta_U W_i)}{m(A)}.
 \]
 Then $A \subseteq \delta U$ and $|A|<+\infty$. Moreover, for any $j>i$,
 \[
  \alpha_{DS}^U(W_i) = \frac{\mathrm{Cap}_{U}(A, \delta_UW_i)}{m(A)} \geq \frac{\mathrm{Cap}_{U}(A, \delta_UW_j)}{m(A)},
 \]
 which means that 
 \[
  \alpha_{DS}^U(W_i) \geq \lim_{j \to \infty} \frac{\mathrm{Cap}_{U}(A, \delta_U W_j)}{m(A)} = \frac{\mathrm{Cap}^U(A)}{m(A)} \geq \alpha_S(U).
 \]
 Thus, $\lim_{i \to \infty}\alpha_{DS}^U(W_i) \geq \alpha_S(U)$. This proves the lemma.
\end{proof}
Now, we are ready to prove Theorem \ref{CheegerInequalityForDtNBottom}.
\begin{proof}[Proof of Theorem \ref{CheegerInequalityForDtNBottom}]
 By applying \cite[Proposition 1.5]{HHW2022}, we have 
 \[
  \lim_{i \to \infty}\sigma_1^D(W_i) = \sigma_1(U).
 \]
 Then the conclusion follows from Lemma \ref{CheegerInequalityForDS} and Lemma \ref{LimitDSConstant}.
\end{proof}

The following example shows that the upper bound estimate in Theorem \ref{CheegerInequalityForDtNBottom} is sharp.
\begin{example}\label{InfiniteTree}
 Consider the graph in Figure \ref{Example} with unit edge weights and vertex weights, which is an induced subgraph of $3$-regular tree. Let $U = \{x_2, x_3, \cdots\}$ and $\delta U = \{x_1\}$. Choose $W_i = \{x_1, x_2, \cdots, x_{2^i}\}$. Obviously, $\{W_i\}_{i=1}^{\infty}$ is an exhaustion of $\overline{U}$ and $\delta_U W_i = \{x_{2^i+1}, \cdots, x_{2^{i+1}}\}$. One can show that $\mathrm{Cap}_{U}(\{x_1\}, \delta_U W_i) = \frac{2^i}{2^{i+1}-1}$. Hence, by Theorem \ref{CheegerInequalityForDtNBottom},
 \[
  \frac18 \leq \sigma_1(U) \leq \frac12.
 \]
 On the other hand, by direct calculation, we have $\sigma_1(U) = \frac12$.
 \begin{figure}
  \centering
   \begin{tikzpicture}
    \filldraw[draw = black, fill = black] (0,0) circle(0.05) (1,0) circle(0.05) (1.3,0.95) circle(0.05) (1.3,-0.95) circle(0.05) (2.2,1.35) circle(0.05);
    \filldraw[draw = black, fill = black] (2.2,-1.35) circle(0.05) (2.2,0.55) circle(0.05) (2.2,-0.55) circle(0.05);
    \foreach \i in {1.55,1.15,0.75,0.35,-0.35,-0.75,-1.15,-1.55}
    \filldraw[draw = black, fill = black] (3.2, \i) circle(0.05);
    \draw (0,0) -- (1,0) -- (1.3,0.95) -- (2.2,1.35) -- (3.2,1.55); \draw (2.2, 0.55) -- (3.2,0.35);
    \draw (1,0) -- (1.3,-0.95) -- (2.2,-1.35) -- (3.2,-1.55); \draw (2.2,1.35) -- (3.2,1.15); \draw (2.2,-0.55) -- (3.2,-0.35);
    \draw (1.3,0.95) -- (2.2,0.55) -- (3.2,0.75); \draw (1.3,-0.95) -- (2.2,-0.55) -- (3.2,-0.75); \draw (2.2,-1.35) -- (3.2,-1.15);
    \draw[densely dashed] (3.2,1.55) -- (4.2,1.7); \draw[densely dashed] (3.2,1.55) -- (4.2,1.4); \draw[densely dashed] (3.2,1.15) -- (4.2,1.3);
    \draw[densely dashed] (3.2,1.15) -- (4.2,1); \draw[densely dashed] (3.2,0.75) -- (4.2,0.9); \draw[densely dashed] (3.2,0.75) -- (4.2,0.6);
    \draw[densely dashed] (3.2,0.35) -- (4.2,0.5); \draw[densely dashed] (3.2,0.35) -- (4.2,0.2);
    \draw[densely dashed] (3.2,-1.55) -- (4.2,-1.7); \draw[densely dashed] (3.2,-1.55) -- (4.2,-1.4); \draw[densely dashed] (3.2,-1.15) -- (4.2,-1.3);
    \draw[densely dashed] (3.2,-1.15) -- (4.2,-1); \draw[densely dashed] (3.2,-0.75) -- (4.2,-0.9); \draw[densely dashed] (3.2,-0.75) -- (4.2,-0.6);
    \draw[densely dashed] (3.2,-0.35) -- (4.2,-0.5); \draw[densely dashed] (3.2,-0.35) -- (4.2,-0.2);
    \node at (0,0.3) {$x_1$}; \node at (1.3,0) {$x_2$}; \node at (1.3,1.25) {$x_3$}; \node at (1.3,-1.25) {$x_4$};
    \node at (2.2,1.05) {$x_5$}; \node at (2.2,0.25) {$x_6$}; \node at (2.2,-0.85) {$x_7$}; \node at (2.2,-1.85) {$x_8$};
   \end{tikzpicture}
  \caption{}
  \label{Example}
 \end{figure}
\end{example}

At the end of this section, we give an example to calculate the bottom of the spectrum of the DtN operator using capacity.
\begin{example}\label{HalfOfZn}
 Consider the lattice graph $(\mathbb{Z}^N, E)$ for $N \geq 3$ with unit vertex weights.
 Let $$U = \{(x_1, \cdots, x_N): x_1, \cdots, x_{N-1} \in \mathbb{Z}, x_N \in \mathbb{Z}_+\},$$ then $$\delta U = \{(x_1, \cdots, x_{N-1}, 0): x_1, \cdots, x_{N-1} \in \mathbb{Z}\}.$$We use Theorem \ref{CheegerInequalityForDtNBottom} to show that $\sigma_1(U) = 0$. For $r \in \mathbb{Z}_+$, we write
 \[
  Q_r:= \{x \in \mathbb{Z}^N: \|x\|_{\infty} \leq r\},
 \]
 and 
 \[
  S_r:= \{x \in \mathbb{Z}^N: \|x\|_{\infty} = r\}.
 \]
 Fix $r_0 \in \mathbb{Z}_+$ and set
 \[
  A = Q_{r_0} \cap \delta U.
 \]
 We define $f \in \mathbb{R}^{\overline{U}}$ as
 \begin{equation*}
  f(x) = 
   \begin{cases}
    1, \quad &x \in Q_{r_0} \cap \overline{U}, \\
    \left(\frac{r_0}{r}\right)^{N-2}, &x \in S_r \cap \overline{U}, r > r_0.
   \end{cases}
 \end{equation*}
 Then $f|_A = 1$, and
 \begin{align*}
  \mathcal{E}_U(f, f) 
  &\leq C(N)\sum_{r \geq r_0}r^{N-1}\left(\left(\frac{r_0}{r}\right)^{N-2}-\left(\frac{r_0}{r+1}\right)^{N-2}\right)^2 \\
  &\leq C(N)r_0^{2N-4}\sum_{r \geq r_0}\frac{1}{r^{N-1}} \\
  &\leq C(N)r_0^{N-2},
 \end{align*}
 where $C(N)$ depends only on $N$. Note that 
 \[
  \mathrm{Cap}^U(A) = \inf\{\mathcal{E}_U(g, g): g \in l_0(\overline{U}), g|_A = 1\},
 \]
 and there exists a sequence of $\{f_i\}_{i=1}^{\infty} \in l_0(\overline{U})$ such that $\mathcal{E}_U(f_i, f_i) \to \mathcal{E}_U(f, f)$, see, e.g., \cite[Theorem 7]{hua2022extremal}, we have
 \[
  \frac{\mathrm{Cap}^U(A)}{m(A)} \leq \frac{\mathcal{E}_U(f,f)}{m(A)} \leq \frac{C(N)}{r_0}.
 \]
 By passing to the limit $r_0 \to +\infty$, we get 
 \[
  \alpha_S(U) = 0.
 \]
 By Theorem \ref{CheegerInequalityForDtNBottom}, we have
 \[
 \sigma_1(U) = 0.
 \]
\end{example}
\begin{remark}
 In \cite{HHW2022}, the first author, Huang and Wang proved that if $U \subseteq V$ is an infinite subset with $|\delta U| < +\infty$, then $\sigma_1(U) = 0$ if and only if $(\overline{U}, E(U, \overline{U}))$ is recurrent. Example \ref{HalfOfZn} shows that if $|\delta U|=\infty$, then $\sigma_1(U) = 0$ does not imply that $(\overline{U}, E(U, \overline{U}))$ is recurrent.
\end{remark}

\section{Higher-order Dirichlet eigenvalues and Steklov eigenvalues}

In this section, we estimate higher-order eigenvalues. Recall that for any subset $W \subseteq V$, $\mathcal{A}(W)$ is the collection of all non-empty finite subsets of $W$, and $\mathcal{A}_k(W)$ is the set of all disjoint $k$-tuples $(A_1, \cdots, A_k)$ such that $A_l \in \mathcal{A}(W)$, for all $l \in \{1, \cdots, k\}$.

We first establish estimates of higher-order Dirichlet eigenvalues on a finite subgraph. Let $G = (V, E, w,m)$ be a graph,  $W \subseteq V$ be a finite subset. Consider the following Dirichlet eigenvalue problem
\begin{equation*}
 \begin{cases}
  \mathcal{L}f(x) = \lambda f(x), \quad &x \in W, \\
  f(x) =0, &x \in \delta W.
 \end{cases}
\end{equation*}
We denote by $\lambda_k^D(W)$ the $k$-th eigenvalue of the above eigenvalue problem. For any $k \in \{1, \cdots, |W|\}$, we define
\[
 \widetilde{\Gamma}_k(W):= \min_{(A_1, \cdots, A_k) \in \mathcal{A}_k(W)}\max_{l \in \{1, \cdots, k\}}\lambda_1^D(A_l).
\]
By \cite[Theorem 5.11]{HHW2022}, there exists a universal constant $c>0$ such that
\[
 \lambda_k^D(W) \geq \frac{c}{k^6}\widetilde{\Gamma}_k(W).
\]
On the other hand, for any $(A_1, \cdots, A_k) \in \mathcal{A}_k(W)$, let $f_i$ be a first eigenfunction on $A_i$, $i = 1, \cdots, k$. Consider $H:=\mathrm{span}\{f_i: i = 1, \cdots, k\}$. Then $\dim H = k$ and
\[
 \max_{0 \neq f \in H}\frac{\mathcal{E}_W(f, f)}{\|f\|_{l^2(W)}^2} \leq 2\max_{l \in \{1, \cdots, k\}}\lambda_1^D(A_l).
\]
It follows that
\[
 \lambda_k^D(W) \leq 2\widetilde{\Gamma}_k(W).
\]
Hence, we get the following useful lemma.
\begin{lemma}\label{HighOrderDirichlet}
 The $k$-th Dirichlet eigenvalue satisfies
 \[
  \frac{c}{k^6}\widetilde{\Gamma}_k(W) \leq \lambda_k^D(W) \leq 2\widetilde{\Gamma}_k(W).
 \]
\end{lemma}

We consider higher-order eigenvalues on an infinite graph. Let $G = (V, E, w, m)$ be an infinite weighted graph, $\mathcal{W} = \{W_i\}_{i=1}^{\infty}$ be an exhaustion. Then for any $k \geq 1$, $\lambda_k(G) = \lim_{i \to \infty}\lambda_k^D(W_i)$, where $\lambda_k^D(W_i)$ is the $k$-th Dirichlet eigenvalue on $W_i$. For each $W_i$, set
\[
 \Gamma_k^D(W_i):= \min_{(A_1, \cdots, A_k) \in \mathcal{A}_k(W_i)}\max_{l \in \{1, \cdots, k\}}\alpha_D(A_l).
\]
By applying Theorem \ref{CheegerInequalityForDirichlet} and Lemma \ref{HighOrderDirichlet}, we have
\begin{equation}\label{CIFHODOF}
 \frac{c}{k^6}\Gamma_k^D(W_i) \leq \lambda_k^D(W_i) \leq 2\Gamma_k^D(W_i).
\end{equation}
\begin{proof}[Proof of Theorem \ref{HigherOrderDirichletEstimate}]
 One easily sees that $\Gamma_k^D(W_i) \geq \Gamma_k^D(W_{i+1})$. Thus, the limit $\lim_{i \to \infty}\Gamma_k^D(W_i)$ exists. We only need to prove that 
 \[
  \lim_{i \to \infty}\Gamma_k^D(W_i) = \Gamma_k^D(G).
 \]
 Recall that 
 \[
  \Gamma_k^D(G) = \inf_{(A_1, \cdots, A_k) \in \mathcal{A}_k(V)}\max_{l \in \{1, \cdots, k\}}\alpha_D(A_l).
 \]
 For any $(A_1, \cdots, A_k) \in \mathcal{A}_k(V)$, there exists $i \in \mathbb{N}_+$ such that $(A_1, \cdots, A_k) \in \mathcal{A}_k(W_i)$. It follows that 
 \[
  \max_{l \in \{1, \cdots, k\}}\alpha_D(A_l) \geq \Gamma_k^D(W_i).
 \]
 Thus,
 \[
  \Gamma_k^D(G) \geq \lim_{i \to \infty}\Gamma_k^D(W_i).
 \]
 On the other hand, let $(A_1, \cdots, A_k) \in \mathcal{A}_k(W_i)$ be a $k$-tuple such that
 \[
  \Gamma_k^D(W_i) = \max_{l \in \{1, \cdots, k\}}\alpha_D(A_l).
 \]
 Then $(A_1, \cdots, A_k) \in \mathcal{A}_k(V)$, and hence,
 \[
  \Gamma_k^D(W_i) \geq \Gamma_k^D(G).
 \]
 It follows that $\lim_{i \to \infty}\Gamma_k^D(W_i) \geq \Gamma_k^D(G)$. This proves the theorem.
\end{proof}

In the following, we discuss higher-order Steklov eigenvalues. First, we consider the case of finite graphs. Let $G = (V, E, w, m)$ be a weighted graph, $\Omega \subseteq V$ be a finite subset. Consider the graph $G_{\Omega}$.

For any $A \in \mathcal{A}(\overline{\Omega})$ satisfying $A \cap \delta \Omega \neq \varnothing$ and $f \in \mathbb{R}^{A \cap \delta \Omega}$, let $u_f^A$ be the solution of the following equation:
\begin{equation*}
 \begin{cases}
  \Delta u_f^A(x) = 0, \quad &x \in A \cap \Omega, \\
  u_f^A(x) = f(x), &x \in A \cap \delta \Omega, \\
  u_f^A(x) = 0, &x \in \overline{\Omega} \setminus A.
 \end{cases}
\end{equation*}
We define the DtN operator on $A$ via
\begin{align*}
 \Lambda_A: \mathbb{R}^{A \cap \delta \Omega} &\to \mathbb{R}^{A \cap \delta \Omega} \\
 f &\mapsto \Lambda_A(f):= \frac{\partial u_f^A}{\partial n},
\end{align*}
where $\frac{\partial}{\partial n}$ is taken w.r.t. $\Omega$. Denote by $\sigma_1^D(A)$ the smallest eigenvalue of $\Lambda_A$. By convention, when $A \cap \delta \Omega = \varnothing$, we set $\sigma_1^D(A) = + \infty$. Then for any $1 \leq k \leq |\delta \Omega|-1$, by \cite[Theorem 5]{HM2020}, there exists a universal constant $c>0$ such that
\begin{equation}\label{HMHighOrderEstimate}
 \frac{c}{k^6}\widetilde{\kappa}_{k+1} \leq \sigma_k(\Omega) \leq 2\widetilde{\kappa}_{k+1},
\end{equation}
where 
\[
 \widetilde{\kappa}_{k+1}:= \min_{(A_1, \cdots, A_{k+1}) \in \mathcal{A}_{k+1}(\overline{\Omega})}\max_{l \in \{1, \cdots, k+1\}}\sigma_1^D(A_l).
\]

Now, we are ready to prove Theorem \ref{HighOrderSteklovOnFinite}.
\begin{proof}[Proof of Theorem \ref{HighOrderSteklovOnFinite}]
 For each $A_l \in \mathcal{A}_{k+1}(\overline{\Omega})$, $l \in \{1, \cdots, k+1\}$, by applying Lemma \ref{CheegerInequalityForDS}, we have
 \[
 \frac14\alpha_{DS}^{\Omega}(A_l) \leq \sigma_1^D(A_l) \leq \alpha_{DS}^{\Omega}(A_l),
\]
where
\[
 \alpha_{DS}^{\Omega}(A_l) = \inf_{A \subseteq A_l \cap \delta \Omega}\frac{\mathrm{Cap}_{\Omega}(A, \delta_{\Omega}A_l)}{m(A)},
\]
and $\delta_{\Omega}A_l$ is the vertex boundary of $A_l$ in $G_{\Omega}$. Then the conclusion follows from \eqref{HMHighOrderEstimate}.
\end{proof}

At the end of this section, we consider the case of infinite subgraph. Let $U \subseteq V$ be an infinite subset with vertex boundary $\delta U$. Consider the graph $G_U$. Let $\mathcal{W}=\{W_i\}_{i=1}^{\infty} \uparrow \overline{U}$ be an exhaustion. Then for any $k \geq 1$, $\sigma_k(U) = \lim_{i \to \infty}\sigma_k^D(W_i)$, where $\sigma_k^D(W_i)$ is the $k$-th eigenvalue of operator $\Lambda_{W_i}$ defined in Section 5, see \eqref{DefinitionOfSigmaKOnWi}.

Fix $W_i$ and set $N = |W_i \cap \delta U|$. We construct a sequence of finite graphs $\{G^{(l)}\}_{l=1}^{\infty}$ where $G^{(l)} = (W_i \cup \delta_U W_i, E(W_i, W_i \cup \delta_U W_i), w, m^{(l)})$ with $m^{(l)}|_{W_i \cap \delta U} = m$ and $m^{(l)}|_{W_i \cap U} = \frac{m}{l}$. Consider the following Dirichlet problem
\begin{equation*}
 \begin{cases}
  -\Delta^{(l)}f(x) = \lambda f(x), \quad &x \in W_i, \\
  f(x) =0, &x \in \delta_U W_i,
 \end{cases}
\end{equation*}
where $\Delta^{(l)}$ is the Laplacian on $G^{(l)}$. Denote by $\lambda_{k, D}^{(l)}(W_i)$ the $k$-th eigenvalue of the above Dirichlet problem. Then according to \cite[Proposition 3]{HM2020}, for any $1 \leq k \leq N$, we have
\[
 \lim_{l \to \infty}\lambda_{k, D}^{(l)}(W_i) = \sigma_k^D(W_i).
\]
Set
\[
 \Gamma^S_k(W_i) = \min_{(A_1, \cdots, A_k) \in \mathcal{A}_k(W_i)}\max_{j \in \{1, \cdots, k\}}\alpha_{DS}^U(A_j).
\]
By applying Lemma \ref{HighOrderDirichlet} and Lemma \ref{CheegerInequalityForDS}, we have
\begin{align*}
 \sigma_k^D(W_i) &=\lim_{l \to \infty}\lambda_{k, D}^{(l)}(W_i) \geq \lim_{l \to \infty}\frac{c}{k^6}\min_{(A_1, \cdots, A_k)\in \mathcal{A}(W_i)}\max_{j \in \{1, \cdots, k\}}\lambda_{1, D}^{(l)}(A_j) \\
 &=\frac{c}{k^6}\min_{(A_1, \cdots, A_k) \in \mathcal{A}_k(W_i)}\max_{j \in \{1, \cdots, k\}}\sigma_1^D(A_j) \\
 &\geq \frac{c}{k^6}\min_{(A_1, \cdots, A_k) \in \mathcal{A}_k(W_i)}\max_{j \in \{1, \cdots, k\}}\frac14\alpha_{DS}^U(A_j) \\
 &=\frac{c'}{k^6}\Gamma^S_k(W_i),
\end{align*}
and
\[
 \sigma_k^D(W_i) = \lim_{l \to \infty}\lambda_{k, D}^{(l)}(W_i) \leq 2\Gamma^S_k(W_i),
\]
i.e., there exists a universal constant $c>0$ such that
\begin{equation}\label{HighOrderEstimateOfExhaustion}
 \frac{c}{k^6}\Gamma^S_k(W_i) \leq \sigma_k^D(W_i) \leq 2\Gamma^S_k(W_i).
\end{equation}

We can prove Theorem \ref{HighOrderSteklovOnInfinite}.
\begin{proof}[Proof of Theorem \ref{HighOrderSteklovOnInfinite}]
 One easily sees that $\Gamma^S_k(W_{i+1}) \leq \Gamma^S_k(W_i)$. Thus, the limit $\lim_{i \to \infty}\Gamma^S_k(W_i)$ exists. We only need to show that
 \[
  \lim_{i \to \infty}\Gamma^S_k(W_i) = \Gamma^S_k(U).
 \]
 Recall that 
 \[
  \Gamma^S_k(U) = \inf_{(A_1, \cdots, A_k) \in \mathcal{A}_k(\overline{U})}\max_{l \in \{1, \cdots, k\}}\alpha_{DS}^U(A_l).
 \]
 For any $(A_1, \cdots, A_k) \in \mathcal{A}_k(\overline{U})$, there exists $i \in \mathbb{N}_+$ such that $(A_1, \cdots, A_k) \in \mathcal{A}_k(W_i)$. Then we have 
 \[
  \max_{l \in \{1, \cdots, k\}}\alpha_{DS}^U(A_l) \geq \Gamma^S_k(W_i),
 \]
 and hence,
 \[
  \Gamma^S_k(U) \geq \lim_{i \to \infty}\Gamma^S_k(W_i).
 \]
 On the other hand, for any $i \in \mathbb{N}_+$, let $(A_1, \cdots, A_k)$ be a $k$-tuple such that 
 \[
  \Gamma^S_k(W_i) = \max_{l \in \{1, \cdots, k\}}\alpha_{DS}^U(A_l).
 \]
 Then $(A_1, \cdots, A_k) \in \mathcal{A}_k(\overline{U})$, and thus,
 \[
  \Gamma^S_k(W_i) \geq \Gamma^S_k(U).
 \]
 It follows that $\lim_{i \to \infty}\Gamma^S_k(W_i) \geq \Gamma^S_k(U)$. Then the conclusion follows from the fact that $\sigma_k(U) = \lim_{i \to \infty}\sigma_k^D(W_i)$ and \eqref{HighOrderEstimateOfExhaustion}.
\end{proof}

\begin{remark}
 One easily sees that
 \[
  \Gamma^S_k(W_i) \geq \alpha_{DS}^U(W_i).
 \]
 Thus, by Lemma \ref{LimitDSConstant}, we have
 \[
  \sigma_{k}(U) \geq \frac{c}{k^6}\alpha_S(U).
 \]
\end{remark}

\section{The DtN operator defined by Hassannezhad and Miclo}

In this section, we briefly discuss the eigenvalue problem of the DtN operator defined in \cite{HM2020}.

We first recall the definition of the DtN operator in \cite{HM2020}. Let $G = (V, E, w, m)$ be a finite graph, $\Omega \subseteq V$ be a proper subset. Given $f \in \mathbb{R}^{\Omega}$, let $u_f$ be its harmonic extension on $V$, namely the unique $u_f \in \mathbb{R}^V$ satisfying
\begin{equation*}
 \begin{cases}
  \Delta u_f(x) = 0, \quad &x \in V\setminus \Omega, \\
  u_f(x) = f(x), &x \in \Omega.
 \end{cases}
\end{equation*}
Then we define the DtN operator $S_{\Omega}$ on $\Omega$ as
\[
 \forall \ x \in \Omega, \quad S_{\Omega}f(x) = -\Delta u_f(x).
\]
For any $1 \leq k \leq |\Omega|-1$, we denote by $\sigma_k^S(\Omega)$ the $k$-th non-trivial eigenvalue of the operator $S_{\Omega}$. Given $f, g \in \mathbb{R}^V$, we define the energy functional via
\[
 \mathcal{E}(f, g):= \frac12\sum_{x, y \in V}w(x, y)(f(y)-f(x))(g(y)-g(x)).
\]
For any $A, B \subseteq V$, we define
\[
 \mathrm{Cap}(A, B):= \inf\{\mathcal{E}(f, f): f|_A = 1, f|_B = 0\}.
\]
Note that Lemma \ref{CoareaFormula} and Theorem \ref{CheegerEstimate} still hold on $G$. For $|\Omega| \geq 2$, set
\[
 \beta_S(\Omega):= \inf_{A, B \subseteq \Omega}\frac{\mathrm{Cap}(A, B)}{m(A) \wedge m(B)}.
\]
Then following the proof of Theorem \ref{CheegerInequalityForSteklov}, we have the following theorem.
\begin{theorem}
 Let $G$ be a finite weighted graph, $\Omega \subseteq V$ be a proper subset with $|\Omega| \geq 2$. Then
 \[
  \frac18\beta_S(\Omega) \leq \sigma_1^S(\Omega) \leq 2\beta_S(\Omega).
 \]
\end{theorem}
For any $1 \leq k \leq |\Omega|-1$, set
\[
 \beta_{k+1}:= \min_{(A_1, \cdots, A_{k+1} \in \mathcal{A}_{k+1}(V)}\max_{l \in \{1, \cdots, k+1\}}\beta_{DS}(A_l),
\]
where 
\[
 \beta_{DS}(A_l):= \inf_{A \subseteq A_l \cap \Omega}\frac{\mathrm{Cap}(A, V \setminus A_l)}{m(A)}.
\]
Then following the proof of Theorem \ref{HighOrderSteklovOnFinite}, we have the following result.
\begin{theorem}
 Let $G$ be a finite weighted graph, $\Omega \subseteq V$ be a proper subset. Then there exists a universal constant $c>0$ such that for any $1 \leq k \leq |\Omega|-1$,
 \[
  \frac{c}{k^6}\beta_{k+1} \leq \sigma_k^S(\Omega) \leq 2\beta_{k+1}.
 \]
\end{theorem}

\section*{Acknowledgments}
We thank Matthias Keller for pointing out approximating Neumann eigenvalues by taking the vanishing limit of vertex weights on the boundary. The third author would like to thank the Max Planck Institute for Mathematics in the Sciences for the hospitality. B. Hua is supported by NSFC, no. 12371056, and by Shanghai Science and Technology Program[Project No. 22JC1400100]. 

\vspace{1em}
\noindent \textbf{Conflict of Interest: }The authors have no competing interests to declare that are relevant to the content of this article. 

\vspace{1em}
\noindent \textbf{Data Availability: }No data, models, or code were used during the study.

\bibliographystyle{plain}
\bibliography{CheegerEstimate}

\begin{thebibliography}{10}

\bibitem{AM1985}
N.~Alon and V.~D. Milman.
\newblock {$\lambda_1,$} isoperimetric inequalities for graphs, and
  superconcentrators.
\newblock {\em J. Combin. Theory Ser. B}, 38(1):73--88, 1985.

\bibitem{BHJ2014}
Frank Bauer, Bobo Hua, and J\"{u}rgen Jost.
\newblock The dual {C}heeger constant and spectra of infinite graphs.
\newblock {\em Adv. Math.}, 251:147--194, 2014.

\bibitem{BKW2015}
Frank Bauer, Matthias Keller, and Rados\l aw~K. Wojciechowski.
\newblock Cheeger inequalities for unbounded graph {L}aplacians.
\newblock {\em J. Eur. Math. Soc. (JEMS)}, 17(2):259--271, 2015.

\bibitem{BH2012}
Andries~E. Brouwer and Willem~H. Haemers.
\newblock {\em Spectra of graphs}.
\newblock Universitext. Springer, New York, 2012.

\bibitem{Calder1980}
Alberto-P. Calder\'{o}n.
\newblock On an inverse boundary value problem.
\newblock In {\em Seminar on {N}umerical {A}nalysis and its {A}pplications to
  {C}ontinuum {P}hysics ({R}io de {J}aneiro, 1980)}, pages 65--73. Soc. Brasil.
  Mat., Rio de Janeiro, 1980.

\bibitem{Cheeger1970}
Jeff Cheeger.
\newblock A lower bound for the smallest eigenvalue of the {L}aplacian.
\newblock In {\em Problems in analysis ({S}ympos. in honor of {S}alomon
  {B}ochner, {P}rinceton {U}niv., {P}rinceton, {N}.{J}., 1969)}, pages
  195--199. Princeton Univ. Press, Princeton, NJ, 1970.

\bibitem{Chung1989}
F.~R.~K. Chung.
\newblock Diameters and eigenvalues.
\newblock {\em J. Amer. Math. Soc.}, 2(2):187--196, 1989.

\bibitem{Chung1997}
Fan R.~K. Chung.
\newblock {\em Spectral graph theory}, volume~92 of {\em CBMS Regional
  Conference Series in Mathematics}.
\newblock Conference Board of the Mathematical Sciences, Washington, DC; by the
  American Mathematical Society, Providence, RI, 1997.

\bibitem{CEG2011}
Bruno Colbois, Ahmad El~Soufi, and Alexandre Girouard.
\newblock Isoperimetric control of the {S}teklov spectrum.
\newblock {\em J. Funct. Anal.}, 261(5):1384--1399, 2011.

\bibitem{CGGS2024}
Bruno Colbois, Alexandre Girouard, Carolyn Gordon, and David Sher.
\newblock Some recent developments on the {S}teklov eigenvalue problem.
\newblock {\em Rev. Mat. Complut.}, 37(1):1--161, 2024.

\bibitem{CDS1995}
Drago\v{s}~M. Cvetkovi\'{c}, Michael Doob, and Horst Sachs.
\newblock {\em Spectra of graphs}.
\newblock Johann Ambrosius Barth, Heidelberg, third edition, 1995.
\newblock Theory and applications.

\bibitem{EBDavies1990}
E.~B. Davies.
\newblock {\em Heat kernels and spectral theory}, volume~92 of {\em Cambridge
  Tracts in Mathematics}.
\newblock Cambridge University Press, Cambridge, 1990.

\bibitem{EBDavis1995}
E.B. Davies.
\newblock {\em Spectral Theory and Differential Operators}.
\newblock Cambridge Studies in Advanced Mathematics. Cambridge University
  Press, 1995.

\bibitem{DK1986}
J.~Dodziuk and W.~S. Kendall.
\newblock Combinatorial {L}aplacians and isoperimetric inequality.
\newblock In {\em From local times to global geometry, control and physics
  ({C}oventry, 1984/85)}, volume 150 of {\em Pitman Res. Notes Math. Ser.},
  pages 68--74. Longman Sci. Tech., Harlow, 1986.

\bibitem{Dodziuk1984}
Jozef Dodziuk.
\newblock Difference equations, isoperimetric inequality and transience of
  certain random walks.
\newblock {\em Trans. Amer. Math. Soc.}, 284(2):787--794, 1984.

\bibitem{Escobar1997}
Jos\'{e}~F. Escobar.
\newblock The geometry of the first non-zero {S}tekloff eigenvalue.
\newblock {\em J. Funct. Anal.}, 150(2):544--556, 1997.

\bibitem{Escobar1999}
Jos\'{e}~F. Escobar.
\newblock An isoperimetric inequality and the first {S}teklov eigenvalue.
\newblock {\em J. Funct. Anal.}, 165(1):101--116, 1999.

\bibitem{FS2011}
Ailana Fraser and Richard Schoen.
\newblock The first {S}teklov eigenvalue, conformal geometry, and minimal
  surfaces.
\newblock {\em Adv. Math.}, 226(5):4011--4030, 2011.

\bibitem{FS2016}
Ailana Fraser and Richard Schoen.
\newblock Sharp eigenvalue bounds and minimal surfaces in the ball.
\newblock {\em Invent. Math.}, 203(3):823--890, 2016.

\bibitem{Fuji1996}
Koji Fujiwara.
\newblock The {L}aplacian on rapidly branching trees.
\newblock {\em Duke Math. J.}, 83(1):191--202, 1996.

\bibitem{GL2021}
Alexandre Girouard and Jean Lagac\'{e}.
\newblock Large {S}teklov eigenvalues via homogenisation on manifolds.
\newblock {\em Invent. Math.}, 226(3):1011--1056, 2021.

\bibitem{HH2023}
Wen Han and Bobo Hua.
\newblock Steklov eigenvalue problem on subgraphs of integer lattices.
\newblock {\em Comm. Anal. Geom.}, 31(2):343--366, 2023.

\bibitem{Hassannezhad2011}
Asma Hassannezhad.
\newblock Conformal upper bounds for the eigenvalues of the {L}aplacian and
  {S}teklov problem.
\newblock {\em J. Funct. Anal.}, 261(12):3419--3436, 2011.

\bibitem{HM2020}
Asma Hassannezhad and Laurent Miclo.
\newblock Higher order {C}heeger inequalities for {S}teklov eigenvalues.
\newblock {\em Ann. Sci. \'{E}c. Norm. Sup\'{e}r. (4)}, 53(1):43--88, 2020.

\bibitem{HH2022}
Zunwu He and Bobo Hua.
\newblock Upper bounds for the {S}teklov eigenvalues on trees.
\newblock {\em Calc. Var. Partial Differential Equations}, 61(3):Paper No. 101,
  15, 2022.

\bibitem{HH2018}
Bobo Hua and Yan Huang.
\newblock Neumann {C}heeger constants on graphs.
\newblock {\em J. Geom. Anal.}, 28(3):2166--2184, 2018.

\bibitem{HHW2017}
Bobo Hua, Yan Huang, and Zuoqin Wang.
\newblock First eigenvalue estimates of {D}irichlet-to-{N}eumann operators on
  graphs.
\newblock {\em Calc. Var. Partial Differential Equations}, 56(6):Paper No. 178,
  21, 2017.

\bibitem{HHW2022}
Bobo Hua, Yan Huang, and Zuoqin Wang.
\newblock Cheeger estimates of {D}irichlet-to-{N}eumann operators on infinite
  subgraphs of graphs.
\newblock {\em J. Spectr. Theory}, 12(3):1079--1108, 2022.

\bibitem{HKSW2023}
Bobo Hua, Matthias Keller, Michael Schwarz, and Melchior Wirth.
\newblock Sobolev-type inequalities and eigenvalue growth on graphs with finite
  measure.
\newblock {\em Proc. Amer. Math. Soc.}, 151(8):3401--3414, 2023.

\bibitem{hua2022extremal}
Bobo Hua, Ruowei Li, and Florentin M{\"u}nch.
\newblock Extremal functions for the second-order sobolev inequality on groups
  of polynomial growth.
\newblock {\em arXiv:2205.00150}, 2022.

\bibitem{HW2020}
Bobo Hua and Lili Wang.
\newblock Dirichlet {$p$}-{L}aplacian eigenvalues and {C}heeger constants on
  symmetric graphs.
\newblock {\em Adv. Math.}, 364:106997, 34, 2020.

\bibitem{H2019}
Hao Huang.
\newblock Induced subgraphs of hypercubes and a proof of the sensitivity
  conjecture.
\newblock {\em Ann. of Math. (2)}, 190(3):949--955, 2019.

\bibitem{Jammes2015}
Pierre Jammes.
\newblock Une in\'{e}galit\'{e} de {C}heeger pour le spectre de {S}teklov.
\newblock {\em Ann. Inst. Fourier (Grenoble)}, 65(3):1381--1385, 2015.

\bibitem{JTYZZ2021}
Zilin Jiang, Jonathan Tidor, Yuan Yao, Shengtong Zhang, and Yufei Zhao.
\newblock Equiangular lines with a fixed angle.
\newblock {\em Ann. of Math. (2)}, 194(3):729--743, 2021.

\bibitem{LU2001}
Matti Lassas and Gunther Uhlmann.
\newblock On determining a {R}iemannian manifold from the
  {D}irichlet-to-{N}eumann map.
\newblock {\em Ann. Sci. \'{E}cole Norm. Sup. (4)}, 34(5):771--787, 2001.

\bibitem{LGT2014}
James~R. Lee, Shayan~Oveis Gharan, and Luca Trevisan.
\newblock Multiway spectral partitioning and higher-order {C}heeger
  inequalities.
\newblock {\em J. ACM}, 61(6):Art. 37, 30, 2014.

\bibitem{Liu2015}
Shiping Liu.
\newblock Multi-way dual {C}heeger constants and spectral bounds of graphs.
\newblock {\em Adv. Math.}, 268:306--338, 2015.

\bibitem{MSS2015}
Adam~W. Marcus, Daniel~A. Spielman, and Nikhil Srivastava.
\newblock Interlacing families {I}: {B}ipartite {R}amanujan graphs of all
  degrees.
\newblock {\em Ann. of Math. (2)}, 182(1):307--325, 2015.

\bibitem{M1964}
Vladimir Maz'ya.
\newblock On the theory of the higher-dimensional {S}chr\"{o}dinger operator.
\newblock {\em Izv. Akad. Nauk SSSR Ser. Mat.}, 28:1145--1172, 1964.

\bibitem{M2009}
Vladimir Maz'ya.
\newblock Integral and isocapacitary inequalities.
\newblock In {\em Linear and complex analysis}, volume 226 of {\em Amer. Math.
  Soc. Transl. Ser. 2}, pages 85--107. Amer. Math. Soc., Providence, RI, 2009.

\bibitem{M1962}
Vladimira Maz'ya.
\newblock The negative spectrum of the higher-dimensional {S}chr\"{o}dinger
  operator.
\newblock {\em Dokl. Akad. Nauk SSSR}, 144:721--722, 1962.

\bibitem{Mohar1982}
Bojan Mohar.
\newblock The spectrum of an infinite graph.
\newblock {\em Linear Algebra Appl.}, 48:245--256, 1982.

\bibitem{MW1989}
Bojan Mohar and Wolfgang Woess.
\newblock A survey on spectra of infinite graphs.
\newblock {\em Bull. London Math. Soc.}, 21(3):209--234, 1989.

\bibitem{Perrin2019}
H\'{e}l\`ene Perrin.
\newblock Lower bounds for the first eigenvalue of the {S}teklov problem on
  graphs.
\newblock {\em Calc. Var. Partial Differential Equations}, 58(2):Paper No. 67,
  12, 2019.

\bibitem{Perrin2021}
H\'{e}l\`ene Perrin.
\newblock Isoperimetric upper bound for the first eigenvalue of discrete
  {S}teklov problems.
\newblock {\em J. Geom. Anal.}, 31(8):8144--8155, 2021.

\bibitem{RVW2002}
Omer Reingold, Salil Vadhan, and Avi Wigderson.
\newblock Entropy waves, the zig-zag graph product, and new constant-degree
  expanders.
\newblock {\em Ann. of Math. (2)}, 155(1):157--187, 2002.

\bibitem{SS2019}
Andr\'{e} Schlichting and Martin Slowik.
\newblock Poincar\'{e} and logarithmic {S}obolev constants for metastable
  {M}arkov chains via capacitary inequalities.
\newblock {\em Ann. Appl. Probab.}, 29(6):3438--3488, 2019.

\bibitem{TH2018}
Francesco Tudisco and Matthias Hein.
\newblock A nodal domain theorem and a higher-order {C}heeger inequality for
  the graph {$p$}-{L}aplacian.
\newblock {\em J. Spectr. Theory}, 8(3):883--908, 2018.

\bibitem{Uhlmann2014}
Gunther Uhlmann.
\newblock Inverse problems: seeing the unseen.
\newblock {\em Bull. Math. Sci.}, 4(2):209--279, 2014.

\bibitem{Woess2000}
Wolfgang Woess.
\newblock {\em Random walks on infinite graphs and groups}, volume 138 of {\em
  Cambridge Tracts in Mathematics}.
\newblock Cambridge University Press, Cambridge, 2000.

\end{thebibliography}

\end{document}